\documentclass{article}     

\usepackage{enumitem}
\usepackage[T1]{fontenc}
\usepackage[utf8]{inputenc}
\usepackage[english]{babel}
\usepackage[T1]{fontenc}
\usepackage{amsfonts}
\usepackage{amsmath}
\usepackage{amssymb}
\usepackage{amstext}
\usepackage{amsthm}
\usepackage{amscd}
\usepackage{cite}
\usepackage{url}
\usepackage{mathtools}
\usepackage{geometry}
\usepackage[hidelinks]{hyperref}
\RequirePackage{fix-cm}
\usepackage{graphicx}

\newtheorem{theorem}{Theorem}[section]
\newtheorem{corollary}[theorem]{Corollary}
\newtheorem{lemma}[theorem]{Lemma}
\newtheorem{proposition}[theorem]{Proposition}

\theoremstyle{definition}
\newtheorem{definition}[theorem]{Definition}
\newtheorem{remark}[theorem]{Remark}

\newtheorem{example}[theorem]{Example}
\newcommand{\Q}{\mathbb{Q}}
\newcommand{\J}{\mathbb{J}}
\newcommand{\s}{\mathbb{S}}
\newcommand{\R}{\mathbb{R}}
\newcommand{\OO}{\mathcal{O}}
\newcommand{\Z}{\mathbb{Z}}

\newcommand{\C}{\mathbb{C}}
\newcommand{\I}{\mathbb{I}}
\newcommand{\B}{\mathfrak{B}}
\newcommand{\A}{\mathfrak{A}}
\newcommand{\p}{\mathfrak{P}}
\newcommand{\q}{\mathfrak{q}}
\newcommand{\N}{\mathbb{N}}
\newcommand{\eps}{\varepsilon}
\newcommand{\abs}[1]{\left|#1\right|}

\newcommand{\Ann}{\text{Ann}}

\def\End{\operatorname{End}}
\newcommand{\Norm}[1]{\left\lVert#1\right\rVert}
\def\Gal{\operatorname{Gal}}

\def\Ker{\operatorname{Ker}}

\def\min{\operatorname{min}}

\def\Cl{\operatorname{Cl}}

\begin{document}
	\title{Primitive divisors of sequences associated to elliptic
		curves with complex multiplication}
\author{Matteo Verzobio}
\date{}
\maketitle

\begin{abstract}
	Let $P$ and $Q$ be two points on an elliptic curve defined over a number field $K$. For $\alpha\in \End(E)$, define $B_\alpha$ to be the $\OO_K$-integral ideal generated by the denominator of $x(\alpha(P)+Q)$. Let $\OO$ be a subring of $\End(E)$, that is a Dedekind domain. We will study the sequence $\{B_\alpha\}_{\alpha\in \OO}$. We will show that, for all but finitely many $\alpha\in \OO$, the ideal $B_\alpha$ has a primitive divisor when $P$ is a non-torsion point and there exist two endomorphisms $g\neq 0$ and $f$ so that $f(P)=g(Q)$. This is a generalization of previous results on elliptic divisibility sequences.
\let\thefootnote\relax\footnotetext{Keywords: Elliptic curves, primitive divisors, elliptic divisibility sequences, endomorphisms, complex multiplication.}
\end{abstract}

	\section{Introduction}
		Let $E$ be an elliptic curve defined over a number field $K$ by a Weierstrass equation with integer coefficients. Consider $P$ and $Q$ two points in $E(K)$ and $\alpha\in \End(E)$, the ring of endomorphisms of $E$. Let $\OO_K$ be the ring of integers of $K$ and consider the fractional $\OO_K$-ideal $(x(\alpha(P)+Q))\OO_K$. We can write in a unique way
		\[
		(x(\alpha(P)+Q))\OO_K=\frac{A_\alpha(P,Q)}{B_\alpha(P,Q)}
		\]
		with $A_\alpha(P,Q)$ and $B_\alpha(P,Q)$ two coprime integral $\OO_K$-ideals. If $\alpha(P)+Q=O$, the identity of the curve, then we put $B_\alpha(P,Q)=0$. Let $\OO$ be a subring of $\End(E)$. We are interested in the study of the sequence $\{B_{\alpha}(P,Q)\}_{\alpha\in \OO}$. This is a sequence of integral $\OO_K$-ideals that depend from the number field $K$, from the elliptic curve $E$, from the points $P$ and $Q$, and from the ring $\OO$. 
		
		Usually, a sequence is a set of element indexed by $\N$. In this paper, we call sequences some sets that are indexed by a subring of $\End(E)$, that can be $\Z$ or a subring of the ring of integers of an imaginary quadratic field. In both cases, our sets of indexes are $\Z$-lattices (of rank $1$ or $2$). Hence, in some sense, they are a close generalization of $\N$. So, with a little abuse of notation, we decided to write that $\{B_\alpha(P,Q)\}_{\alpha \in \OO}$ is a sequence.
		\begin{definition}
			Fix $\OO$ a subring of $\End(E)$. We say that a term $B_\alpha(P,Q)$ of the sequence $\{B_{\beta}(P,Q)\}_{\beta \in \OO}$ has a primitive divisor if there exists a prime $\OO_K$-ideal $\p$ so that $\p$ divides $B_\alpha(P,Q)$ and does not divide $B_\beta(P,Q)$ for every $\beta\in \OO$ such that $0\leq \Norm{\beta}<\Norm{\alpha}$ and $B_\beta(P,Q)\neq 0$. With $\Norm{\cdot}$ we denote the degree of the endomorphism.
		\end{definition}
	In the case when $\OO=\Z$ and $Q=O$, the sequences $\{B_k(P,O)\}_{k\in \Z}$ are the so-called elliptic divisibility sequences. These sequences have been studied extensively and have many applications. 
	
	The problem of the primitive divisors for elliptic divisibility sequences has been introduced by Silverman in 1988.
			\begin{theorem}\cite[Proposition 10, Silverman, 1988]{silverman}
			Let $E$ be an elliptic curve defined over $\Q$, in minimal global form. Take $P$ in $E(\Q)$. Consider the sequence $\{B_n(P,O)\}_{n\in \Z}$. Suppose that $P$ is a non-torsion point. Then, for all but finitely many $n\in \Z$, $B_n(P,O)$ has a primitive divisor.
		\end{theorem}
	This theorem was generalized by Cheon and Hahn, in 1999, to elliptic curves defined over any number fields.
		\begin{theorem}\cite[Cheon and Hahn, 1999]{ChHa}
		Let $E$ be an elliptic curve defined over a number field $K$. Take $P$ in $E(K)$. Consider the sequence $\{B_n(P,O)\}_{n\in \Z}$. Suppose that $P$ is a non-torsion point. Then, for all but finitely many $n\in \Z$, $B_n(P,O)$ has a primitive divisor.
	\end{theorem}
	 The previous theorem was generalized by Streng in \cite{streng} to the case when $\OO$ is not $\Z$.
\begin{theorem}\cite[Principal case of the main theorem, Streng, 2008]{streng}\label{intstreng}
Let $E$ be an elliptic curve defined over a number field $K$ by a Weierstrass equation with coefficients in $\OO_K$. Take $P$ in $E(K)$. Let $\OO$ be a subring of $\End(E)$ and consider the sequence $\{B_\alpha(P,O)\}_{\alpha\in \OO}$. Suppose that $P$ is a non-torsion point. Then, for all but finitely many $\alpha\in \OO$, $B_\alpha(P,O)$ has a primitive divisor.
\end{theorem}
The case when $Q\neq O$ has been introduced in \cite{shparlinski}.
\begin{theorem}\cite[Theorem 1.1, Everest and Shparlinski, 2005]{shparlinski}
Let $E$ be an elliptic curve defined over a number field $K$ by a Weierstrass equation with coefficients in $\OO_K$. Take $P$ and $Q$ two points on $E(K)$. Consider the sequence $\{B_n(P,Q)\}_{n\in \Z}$. Suppose that $P$ is a non-torsion point. Let $\omega_K(I)$ be the function that counts the number of distinct prime divisors of the $\OO_K$-ideal $I$. Then, for all $M$ large enough,
\[
\omega_K\left(\prod_{i=1}^M(B_k(P,Q))\right)\gg M.
\]
\end{theorem}
This theorem suggests that it should be possible to generalize the results stated before on the primitive divisors also in the case when $Q\neq O$.

The case when $Q\neq O$ is a torsion point of order prime, has been studied in \cite{shifted}.
\begin{theorem}\cite[Theorem 1.3, V., 2020]{shifted}\label{intrshift}
		Let $E$ be an elliptic curve defined over a number field $K$ by a Weierstrass equation with coefficients in $\OO_K$. Take $P$ and $Q \in E(K)$. Consider the sequence $\{B_n(P,Q)\}_{n\in \N}$. Suppose that $P$ is a non-torsion point and that $Q$ is a torsion point of prime order. Then, for all but finitely many $n\in \N$, $B_n(P,Q)$ has a primitive divisor.
\end{theorem} 
The study of the sequences $\{B_n(P,Q)\}_{n\in \N}$ is related to a Lang-Trotter conjecture, as is shown in \cite[Section 4]{shifted}. Indeed, the study of the divisors of $B_n(P,Q)$ tells us some information on the orbit of $P$ and $Q$ in the curve reduced modulo every prime.

In this paper, we will generalize the work of Streng in \cite{streng} to the case when $Q\neq O$. We were not able to generalize the result of Streng for $P$ and $Q$ generic and we will need to add some hypotheses on the points. Moreover, we will generalize Theorem \ref{intrshift}. 

Indeed, we prove the following theorems.

	\begin{theorem}\label{Thm1}
		Let $E$ be an elliptic curve defined over a number field $K$ by a Weierstrass equation with coefficients in $\OO_K$. Take $P$ and $Q \in E(K)$. Let $\OO$ be a subring of $\End(E)$. Assume that $\OO$ is a Dedekind domain and consider the sequence $\{B_\alpha(P,Q)\}_{\alpha\in \OO}$. Suppose that $Q$ is a torsion point and that $P$ is a non-torsion point. Then, for all but finitely many $\alpha\in \OO$, $B_\alpha(P,Q)$ has a primitive divisor.
\end{theorem}
Then, in Section \ref{f=g}, we will generalize this result.
\begin{theorem}\label{Thm2}
		Let $E$ be an elliptic curve defined over a number field $K$ by a Weierstrass equation with coefficients in $\OO_K$. Take $P$ and $Q \in E(K)$. Let $\OO$ be a subring of $\End(E)$. Assume that $\OO$ is a Dedekind domain and consider the sequence $\{B_\alpha(P,Q)\}_{\alpha\in \OO}$. Suppose that $P$ is a non-torsion point and that there exist $g\neq 0$ and $f$ in $\OO$ so that $f(P)=g(Q)$. Then, for all but finitely many $\alpha\in \OO$, $B_\alpha(P,Q)$ has a primitive divisor.
\end{theorem}
Observe that Theorem \ref{Thm2} is a generalization of Theorem \ref{Thm1} since, if $f=0$, then $Q$ is a torsion point and so we are in the hypothesis of Theorem \ref{Thm1}.

 In the end, we show how these results generalize Theorem \ref{intrshift}, proving the following theorem, which is a special case of Theorem \ref{Thm1}. 
\begin{theorem}\label{Thm3}
Let $E$ be an elliptic curve defined over a number field $K$ by a Weierstrass equation with coefficients  in $\OO_K$. Take $P$ and $Q$ in $E(K)$. Consider the sequence $\{B_n(P,Q)\}_{n\in \Z}$. Suppose that $P$ is a non-torsion point and that $Q$ is a torsion point. Then, for all but finitely many $n\in \Z$, $B_n(P,Q)$ has a primitive divisor.
\end{theorem}
Observe that Theorem \ref{Thm3} is a generalization of Theorem \ref{intrshift} since we remove the hypothesis that the order of $Q$ is prime.
\begin{remark}
	If $P$ is a torsion point, then $\{B_\alpha(P,Q)\}_{\alpha\in \OO}$ is a periodic sequence. For such sequences, it is clear that there is no a primitive divisor for all but finitely many terms. Therefore, the hypothesis $P$ non-torsion is necessary. Moreover, the hypothesis $g\neq 0$ is necessary in order to use Theorem \ref{Thm1} to prove Theorem \ref{Thm2}. If $g=0$, then $f(P)=O$ and so $f=0$ since $P$ is a non-torsion point. So, we would have $P$ and $Q$ generic and we are not able to prove the theorem in this case. Finally, regarding the hypothesis that $\OO$ is a Dedekind domain, we were not able to remove this hypothesis, even if our ideas follow the proof of Theorem \ref{intstreng}, where the hypothesis is not necessary.
\end{remark}
\begin{remark}
	In the introduction of \cite{shparlinski}, it is claimed that there exists a preprint by Everest and King with a proof of our Theorem \ref{Thm3}. This preprint was never published and we were not able to find the preprint. 
	Moreover, in the PhD thesis of King \cite{kingphd}, there is a proof of Theorem \ref{Thm3}, but this proof has a mistake. We show where this mistake is: Let $P$ and $T$ be two rational points on an elliptic curve defined over $\Q$ so that $T$ is a $q$ torsion point. In \cite[beginning of page 70]{kingphd}, it is assumed that, for every prime $p$ such that $nqP\equiv O\mod{p}$, we have $nP+T\equiv O\mod{p}$. This is clearly false. For example, if $T$ is a $3$-torsion point and $p$ is such that $nP-T\equiv O \mod {p}$ but $T \not\equiv O\mod {p}$, then \[nP+T\equiv (nP-T)-T\equiv -T\not\equiv O\mod{p}\] and \[3nP\equiv 3nP-3T\equiv 3(nP-T)\equiv O\mod {p}.\]
	
	For these reasons, we assume that our Theorem \ref{Thm3} is original.
\end{remark}
\begin{remark}
If $E$ does not have complex multiplication, then $\OO=\Z$. In this case, we are almost in the setting of \cite{shifted}. For this reason, we call this paper "sequences associated to curves with complex multiplication". In any case, even when $E$ does not have CM, Theorem \ref{Thm2} and Corollary \ref{corshifter} are new and generalize the work in \cite{shifted}.
\end{remark}

The paper is organized as follows. In Section \ref{prel}, we give some preliminary definitions on elliptic curves. In Section \ref{structure}, we give a short recap of the structure of the proof of Theorem \ref{Thm1}. In the next sections, we begin with the proof of Theorem \ref{Thm1}. We will show that $B_\alpha(P,Q)$ has a primitive divisor for $\Norm{\alpha}$ large enough. In Section \ref{secgoodprime} and \ref{secbadprime}, we study the prime divisors of $B_\alpha(P,Q)$ that are not primitive divisors. We split the study of the divisors of $B_\alpha(P,Q)$ in two cases, the case of the good prime divisors (that have some useful properties) and the case of the bad prime divisors (that do not have the useful properties, but they are finite). Then, in Section \ref{secproof}, we conclude the proof of Theorem \ref{Thm1}, showing that $B_\alpha$ necessarily has a primitive divisor, if $\Norm{\alpha}$ large enough. Finally, in Section \ref{f=g}, we prove Theorem \ref{Thm2} and Theorem \ref{Thm3}, using Theorem \ref{Thm1}.
\section{Preliminaries}\label{prel}
Let $E$ be an elliptic curve defined by the equation
\[
y^2+a_1xy+a_3y=x^3+a_2x^2+a_4x+a_6
\]
with the coefficients in $\OO_K$, the ring of integers of a number field $K$.

Let $\OO$ be a subring of $\End(E)$. Let $L$ be the field of quotients of $\OO$. This is a subfield of the field of quotients of $\End(E)$. 
As is proved in \cite[Corollary III.9.4]{arithmetic}, the ring $\End(E)$ can be either $\Z$ or an order in an imaginary quadratic field. In the second case, we say that $E$ has complex multiplication (or CM). 
Thus, $L$ is a number field of degree at most $2$. 

The ring $\OO$ is a Dedekind domain and so every non-zero $\OO$-ideal can be written as product of prime $\OO$-ideals. This factorization is unique.

Fix $P$ a non-torsion point and $Q$ a torsion point in $E(K)$. \textbf{From now on, we denote $B_\alpha(P,Q)$ with $B_\alpha$ for any $\alpha\in \OO$.}
\begin{definition}
	Given $\alpha\in \OO$, we denote with $(\alpha)$ the $\OO$-ideal generated by $\alpha$.
\end{definition}
\begin{definition}
	Let $\p$ be a prime ideal in $\OO_K$. We say that $\p$ is a \textbf{primitive divisor} of $B_\alpha$ if $\p$ divides $B_\alpha$ and does not divide $B_\beta$ for every $\beta\in \OO$ such that $0\leq \Norm{\beta}<\Norm{\alpha}$ and $B_\beta\neq 0$. Recall that, as we said in the introduction, with $\Norm{\cdot}$ we denote the degree of the endomorphism. For a definition of degree, see \cite[Section II.2]{arithmetic}. Conversely, we say that $\p$ is a \textbf{non-primitive divisor} of $B_\alpha$ if $\p$ divides $B_\alpha$ but it is not a primitive divisor.
\end{definition}
\begin{lemma}
	Let $K_1$ be a finite extension of $K$ and consider the sequence $\{B_\alpha\OO_{K_1}\}_{\alpha \in \OO}$ of integral $\OO_{K_1}$-ideals. If $B_\alpha\OO_{K_1}$ has a primitive divisor, then $B_\alpha$ has a primitive divisor.
\end{lemma}
\begin{proof}
	Let $\p'$ be a primitive divisor of $B_\alpha\OO_{K_1}$ and take $\p=\p'\cap \OO_K$. So, $\p$ divides $B_\alpha$. If $\p$ is not a primitive divisor of $B_\alpha$, then $\p$ divides $B_\beta$ for $\Norm{\beta}<\Norm{\alpha}$ with $B_\beta\neq 0$. Hence, $\p\OO_{K_1}$ divides $B_\beta\OO_{K_1}$. But $\p'$ divides $\p\OO_{K_1}$ and then $B_\beta\OO_{K_1}$, that is absurd. So, $\p$ is a primitive divisor for $B_\alpha$. 
\end{proof}
Thanks to the previous lemma, we can prove the theorem for $K_1$, a finite extension of $K$. So, during the proof, we can substitute $K$ with a larger number field. 
\begin{definition}\label{defS}
	Recall that $Q$ is a torsion point. Define the non-zero $\OO$-ideal \[\s\coloneqq\{s\in \OO\mid sQ=O\}.\]
\end{definition}
Given a prime $\OO_K$-ideal $\p$ and $R_1,R_2\in E(K)$, we say that \[R_1\equiv R_2\mod{\p}\] if $R_1$ and $R_2$ reduce to the same point in the reduced curve modulo $\p$. It is easy to show that, given $R\in E(K)$, we have $R\equiv O\mod{\p}$ if and only if $\nu(x(R))<0$, where $\nu$ is the place associated with $\p$.
\begin{definition}\label{Ip}
	Take $\p$ a prime $\OO$-ideal and let $R\in E(K)$. Define the integral $\OO$-ideal
	\[
	\Ann_{\p}(R)=\{i\in \OO\mid iR\equiv O\mod{\p}\}.
	\]
\end{definition} 
\begin{lemma}\label{finiteS}
	For all but finitely many prime $\OO$-ideals $\p$,  \[\Ann_{\p}(Q)=\s.\] 
\end{lemma}  
\begin{proof}
	If $E$ is not singular modulo $\p$ and $\p$ is coprime with $\s$, then the reduction modulo $\p$ of the $\s$-torsion points is injective. This is proved in \cite[Proposition VII.3.1]{arithmetic}. Thus, for these primes, the point $Q$ reduces to an $\s$-torsion point modulo $\p$.
\end{proof}

Now, we introduce the theory of heights for elliptic curves. For the details, see \cite[Chapter VIII]{arithmetic}. Let $M_K$ be the set of places of $K$ normalized so that
\[
\prod_{\nu\in M_K}\abs{x}_\nu^{n_\nu}=1
\] 
for every $x\in K^*$, where $n_\nu=[K_\nu:\Q_\nu]$. Given $\nu$ in $M_K$, let \[h_\nu(x)=\max\{0,\log\abs{x}_\nu\}\] for $x\in K^*$ and
\[
h(x)=\frac{1}{[K:\Q]}\sum_{\nu\in M_K}n_\nu h_\nu(x).
\]
\begin{definition}
	For any $\nu\in M_K$ and $R\in E(K)\setminus\{O\}$, define  \[h_\nu(R)=h_\nu(x(R)).\]Moreover, for $R\in E(K)\setminus\{O\}$, we define the height of the point as \[h(R)=h(x(R)).\]  Finally, we put
	\[
	h(O)=0,
	\]
	where $O$ is the identity of the curve.
\end{definition} 
If $\nu$ is finite, then $h_\nu(R)>0$ if and only if $R\equiv O\mod{\p}$, where $\p$ is the prime associated to $\nu$. 

\begin{definition}
Let $\hat{h}(R)$ be the canonical height as defined in \cite[Section VIII.9]{arithmetic}. For every $R\in E(K)$, we put
\[
\hat{h}(R)=\frac 12\lim_{n\to \infty}\frac{h(2^nR)}{4^n}.
\]
The fact that this limit exists is proved in \cite[Proposition VIII.9.1]{arithmetic}.
\end{definition}

\begin{lemma}\label{canheight}
	The canonical height has the following properties:
	\begin{itemize}
		\item $\hat{h}(R)=0$ if and only if $R$ is a torsion point. Otherwise, $\hat{h}(R)>0$.
		\item For every $R\in E(K)$, \[\abs{h(R)-2\hat{h}(R)}\leq C_E,\] with $C_E$ that depends only on $E$.
		\item If $f:E_1\to E_2$ is an isogeny between two elliptic curves and $R\in E_1(K)$, then
		\[
		\hat{h}(f(R))=\Norm f\hat{h}(R),
		\] 
		where with $\Norm f$ we mean the degree of the isogeny. 
	\end{itemize}  
\end{lemma}
\begin{proof}
	The first two facts follow from \cite[Theorem VIII.9.3]{arithmetic}. The last one is proved in \cite[Lemma 3.1]{thesisstreng}.
\end{proof}
  
\begin{definition}
	Given $\alpha\in \OO$, let $E[\alpha]$ be the subgroup of points of $E(\overline{\Q})$ that are in the kernel of $\alpha$. For every $\OO$-ideal $I$, let $E[I]$ be
	\[
	E[I]=\bigcap_{i\in I} E[i].\]
	We say that the points of $E[I]$ are the $I$-torsion points.
\end{definition} 
In the case when $I=(n)$, the group $E[I]$ is the group of the $n$-torsion points.
\begin{definition}\label{EI}
	Define the curve
	\[
	E_I=E/E[I]
	\]
	as in \cite[Proposition III.4.12]{arithmetic}. The elliptic curve $E_I$ has the property that there exists an isogeny $f:E\to E_I$ such that $\ker(f)=E[I]$. 
\end{definition}
Let $\omega_E$ be the invariant differential of $E$ and $\alpha^*\omega_E$ be the pullback along $\alpha$. For the details on the invariant differential, see \cite[Section III.5]{arithmetic}. Enlarging $K$, we can assume that $L=\OO\otimes_\Z\Q$ is a subfield of $K$. We fix the embedding $\OO\subseteq L\hookrightarrow K$ such that
\[
\alpha=\frac{\alpha^*\omega_E}{\omega_E}
\]
for any $\alpha\in \End(E)$. The existence of this embedding is proved in \cite[Proposition II.1.1]{advanced}. We need to fix this embedding in order to apply the results of Streng in \cite{streng}. We will not give more details on the invariant differential since it is just a technical tool used to prove Lemma \ref{315}.

Thanks to Definition \ref{EI}, we know that for every ideal $I$ there exists an isogeny $f$ from $E$ to an elliptic curve $E_I$ such that $\Ker(f)=I$. Since the $\OO_K$-ideals are infinite, we have that the set of elliptic curves $\{E_I\}_{I \text{ ideals}}$ is infinite. During the proof of Theorem \ref{Thm1}, we will need the set $\{E_I\}$ to be finite. In the next lines, choosing the isogeny of Definition \ref{EI} in an appropriate way, we will show that we can assume that the set $\{E_I\}_{I \text{ ideals}}$ is finite. We follow the work of Streng.

For the details on what follows, see \cite[Chapter II]{advanced} or \cite[pages 200-201]{streng}.
\begin{definition}
	Let $\Cl(\OO)$ be the set of $\OO$-ideals modulo equivalence, where we say that $I$ and $J$ are equivalent if there exists $x\in L^*$ such that $I=xJ$. The set $\Cl(\OO)$ is finite.
\end{definition}
If $I$ and $I'$ are in the same class in $\Cl(\OO)$, then $E_I$ is isomorphic to $E_{I'}$. For every class $\overline{I}$ in $\Cl(\OO)$, we fix a curve $E_{\overline{I}}$ (for $E_{\OO}$, we choose $E$). If $\overline{I}$ is the class of $I$ in $\Cl(\OO)$, then there is an isogeny $\varphi_{I}:E\to E_{\overline{I}}$ such that the kernel of the isogeny is $E[I]$. If $(\alpha)=IJ$, then there are two isogenies $\varphi_I$ and $f_J$ so that $[\alpha]=f_J\circ \varphi_I$ and $f_J:E_{\overline{I}}\to E_{(\alpha)}=E$ is an isogeny with kernel equal to the $J$-torsion points of $E_{\overline{I}}$. Since $E[I]$ is $\Gal(\overline{\Q}/K)$-invariant, one can prove (see \cite[Remark III.4.13.2]{arithmetic}) that $E_{\overline{I}}$ can be defined over $K$ for every $\OO$-ideal $I$ and that every isogeny $\varphi_I$ and $f_J$ can be defined over $K$. Even if the choice of the isogenies $\varphi_I$ and $f_J$ is not unique, we fix a choice from now on. Hence, we have a set of isogenies $\varphi_I$ and $f_J$ defined over $K$.
\begin{definition}
	We denote with $\Norm{I}$ the degree of the isogeny $\varphi_{I}$.  Hence, we have $\Norm{\alpha}=\Norm{(\alpha)}$, since $\Norm{\alpha}$ is the degree of the isogeny $\alpha$.
\end{definition}
\begin{definition}
	For every $\overline{I}\in \Cl(\OO)$, we define the fractional $\OO_K$-ideal $C_{\overline{I}}$ following \cite[page 201]{streng}. Let $\omega_{\overline{I}}$ be the invariant differential of $E_{\overline{I}}$. Choose $I$ an element in $\overline{I}$ and put
	\[
	C_{\overline{I}}=\frac{\varphi_I^*\omega_{\overline{I}}}{\omega_E}(I\OO_K)^{-1}.
	\]
	This is a fractional $\OO_K$-ideal, that is independent from the choice of $I$. 
\end{definition}
\begin{lemma}\cite[Lemma 3.15]{streng}\label{315a}
	Let $\nu$ be a finite place and $p$ be the rational prime so that $\nu(p)>0$. Let $t_\nu$ be the smallest integer so that
	\[
	t_\nu\geq \frac{\nu(p)}{p-1}+\nu(C_{\overline{I_1}})-\nu(C_{\overline{I_2}})
	\]
	for all the ideals $I_1$ and $I_2$ in $\OO$. Take a point $P\in E(K)$ and suppose that $\nu(x(\varphi_I(P)))\leq -t_\nu$. If $I\mid J$, then
	\[
	\nu(x(\varphi_J(P)))=\nu(x(\varphi_I(P)))+2\nu(J)-2\nu(I)+2\nu(C_{\overline{J}})-2\nu(C_{\overline{I}}).
	\]
\end{lemma}
\begin{definition}\label{defW}
	Let $W$ be the finite set of places of $K$ (that we call the \textbf{bad places}) composed by the following places:
	\begin{itemize}
		\item the archimedean places;
		\item the places where $E_{\overline{I}}$ has bad reduction for a class $\overline{I}$ in $\Cl(\OO)$;
		\item the places where $\s\neq \Ann_{\p}(Q)$, that are finite as we showed in Lemma \ref{finiteS};
		\item the places that ramify over $\Q$;
		\item the places $\nu$ where $\nu(C_{\overline{I}})\neq 0$, that are finite since $\Cl(\OO)$ is finite. 
	\end{itemize}
	This finite set depends on $E$ and $Q$.
\end{definition}
\begin{lemma}\label{315}
	Let $\nu\notin W$ and $I$ be an integral $\OO$-ideal. Let $J$ be an integral $\OO$-ideal such that $I \mid J$. Let $R\in E(K)$ be such that $\nu(x(\varphi_I(R)))<0$. Then,
	\[
	\nu(x(\varphi_J(R)))=\nu(x(\varphi_I(R)))-2\nu(I)+2\nu(J).
	\] 
\end{lemma}
\begin{proof}
	Let $p$ be the rational prime associated with $\nu$. Since $\nu$ is not ramified, we have $\nu(p)=1$. Moreover, $\nu(C_{\overline{I}})=0$ for every $\OO$-ideal $I$. So, $t_\nu$, as defined in Lemma \ref{315a}, is equal to $1$. Hence, $	\nu(x(\varphi_I(R)))\leq-1\leq -t_\nu$. Therefore, using Lemma \ref{315a},
	\begin{align*}	\nu(x(\varphi_J(P)))&=\nu(x(\varphi_I(P)))-2\nu(J)+2\nu(I)-2\nu(C_{\overline{J}})+2\nu(C_{\overline{I}})\\&=\nu(x(\varphi_I(P)))-2\nu(J)+2\nu(I).\end{align*}
\end{proof}
With the following example, we show that $\Cl(\OO)$ can have cardinality larger than $1$. In particular, we show that it is possible to find an $\OO$-ideal that is not principal. This example shows why we have to consider also the case when $\Cl(\OO)$ is not trivial.
\begin{example}
Let $\alpha=\sqrt{2\sqrt{2}-1}$ and $K=\Q(\alpha)$. Put
\[
j_0=\left[2\left(323+228\left(\frac{\alpha^2+1}2\right)+\alpha\left(231+161\left(\frac{\alpha^2+1}2\right)\right)\right)\right]^3
\]
and consider the elliptic curve $E$ defined by the equation
\[
y^2+xy=x^3-\frac{36}{j_0-1728}x-\frac 1{j_0-1728}.
\]
This curve has $j$-invariant $j_0$ by direct computation and it is defined over $K$. The endomorphism ring of this curve is $\Z[\sqrt{-14}]$. We will not prove this in detail. Anyway, this follows from \cite[Section VI.4]{arithmetic} and \cite[Equation 12.1, page 226]{coxcm}. Put $\OO=\End(E)=\Z[\sqrt{-14}]$. This is a Dedekind domain since it is the ring of integers of the number field $\Q(\sqrt{-14})$. Consider the $\OO$-ideal $I=(2,\sqrt{-14})$. This ideal has norm $2$ and there are no elements of norm $2$ in $\OO$. So, $I$ is not principal and then $\Cl(\OO)$ is not trivial.
\end{example}
During the paper, we will use the following notation. Given three functions $f,g, h:\OO\to \R^{>0}$, we say that 
\[
f(\alpha)=h(\alpha)+O(g(\alpha))
\]
if there exists a constant $C$ such that, for every $\alpha \in \OO$, 
\[
f(\alpha)\leq h(\alpha)+Cg(\alpha).
\]
In the same way, given three functions $f,g,h:\OO\to \R^{>0}$, we say that 
\[
f(\alpha)\leq h(\alpha)+O(g(\alpha))
\]
if there exists a constant $C$ such that, for every $\alpha \in \OO$, 
\[
f(\alpha)\leq h(\alpha)+Cg(\alpha).
\]

We conclude this section with two remarks on the hypothesis that $\OO$ is a Dedekind domain.
\begin{remark}
	In general, it is not true that a subring $\OO$ of $\End(E)$ is a Dedekind domain. We show an example. Let $E$ be the elliptic curve defined by the equation $y^2=x^3-135x-594$. As is shown in the database \cite{lmfdb}, we have $\End(E)=\Z[\sqrt{-3}]$. If $\OO=\Z[\sqrt{-3}]$, then $\OO$ is not a Dedekind domain. Indeed, if $I=(2,1+\sqrt{-3})\OO$, then $I^2=(2)I$ and $I\neq (2)$. Hence, $I$ does not factor uniquely into prime ideals in this ring. Therefore, $\OO$ is not a Dedekind domain.
\end{remark}
\begin{remark}
Observe that if $\OO$ is a Dedekind domain, then it is integrally closed. Hence, it is a maximal order. So, $\OO$ is the ring of integers of the field $L$.
\end{remark}
\section{Structure of the proof}\label{structure}
Assume $\OO=\Z$ and $Q=O$. Observe that $B_{n}(P,O)=B_{-n}(P,O)$. Then $B_n\coloneqq B_n(P,O)$ with $n>0$ has a primitive divisor if there exists a prime $\p$ that divides $B_n$ and does not divide $B_k$ for $0<k<n$. In order to prove that $B_n$ has a primitive divisor for all but finitely many terms, Silverman in \cite{silverman} used the following strategy:
\begin{enumerate}
	\item If a divisor $\p$ of $B_n$ is non-primitive, then it divides $B_k$ for some $k \mid n$.
	\item By results using formal groups, the valuation of $B_n$ at $\p$ is very close to the valuation of $B_k$ at $\p$.
	\item By results from Diophantine approximation, $\log(\abs{B_n})$ is close to $h(nP)$, which grows like $2n^2\hat{h}(P)$.
	\item So, if $B_n$ does not have a primitive divisors, then
	\[
	\log \abs{B_n}\leq \sum_{k\mid n} \log \abs{B_k}+O(\log n)
	\]
	and hence
	\[
	n^2\hat{h}(P)\leq \sum_{k\mid n} k^2\hat{h}(P)+O(n),
	\] which yields a contradiction when $n$ gets large.
\end{enumerate}
To extend this to imaginary quadratic $\OO$, Streng in \cite{streng} defines a divisibility sequence $\B_I$ indexed by ideals $I$ of $\OO$ with $\B_{(\alpha)}=B_\alpha$ for any $\alpha\in \OO$ and replaces the estimate of step 4 by a sharper inclusion-exclusion argument. 

We want to generalize these techniques to our case. Let $\alpha \in \OO$ and we want to show that $B_{\alpha}(P,Q)$ has a primitive divisor if the degree of $\alpha$ is large enough. Our sequences are not divisibility sequences, hence we do not get $k \mid n$ as in step 1 and even the sharper inclusion-exclusion argument will not save us. 
So instead, we construct a strong divisibility sequence $\B_I$, indexed by some ideals of $\OO$, such that $\B_{(\alpha)}=B_\alpha(P,Q)$. If $\p$ is a non-primitive divisor of $B_\alpha$, then $\p$ divides $\B_I$ for $I$ a proper divisor of $(\alpha)$. This will allow us to get a result similar to steps 1 and 2. Then, we will use a modification of step 4 as Streng did in \cite{streng}, in order to prove Theorem \ref{Thm1}.
\section{Study of the good primes}\label{secgoodprime}
Recall that $W$ is the set of bad places as defined in Definition \ref{defW}.
\begin{definition}
	Let $\p$ be a prime $\OO_K$-ideal. We denote with $\nu_\p$ the place associated with $\p$. We say that $\p$ is a \textbf{good prime} if $\nu_\p\notin W$.
\end{definition}
The aim of this section is to study the good primes. Fix $P$ a non-torsion point and $Q$ an $\s$-torsion point, for $\s$ an $\OO$-ideal. We want to study $B_\alpha(P,Q)$ for $\alpha\in \OO$. We are going to study the prime divisors $\p$ of $B_\alpha(P,Q)$ for $\nu_\p\notin W$.

From now on, we fix the curve and the endomorphism $\alpha\in \OO$. Moreover, as we said before, we put $B_\beta= B_\beta(P,Q)$ for every $\beta\in \OO$.
\subsection{Some auxiliary ideals}
\begin{definition}\label{IJK}
	Recall that we defined $\Ann_{\p}(P)$ in Definition \ref{Ip} as the integral $\OO$-ideal
	\[
	\Ann_{\p}(P)=\{i\in \OO\mid iP\equiv O\mod \p\}.
	\]
	Define the fractional $\OO$-ideals
	\[
	I_{\p}(P)\coloneqq\frac{\Ann_{\p}(P)}{\s}
	\]
	and
	\[
	J_{\p}(P)\coloneqq\frac{(\alpha)}{I_{\p}(P)}=\frac{(\alpha)\s}{\Ann_{\p}(P)}.
	\]
	Observe that $J_{\p}(P)$ depends implicitly from $\alpha$.
\end{definition}
The goal of next lemmas is to show that, if $\p$ is a good prime and it is a divisor of $B_{\alpha}$, then $I_{\p}(P)$ and $J_{\p}(P)$ are integral ideals.

We begin with an easy example.
\begin{example}
	Let $E$ be the rational elliptic curve defined by the equation $y^2=x^3-11x+890$. Let $P=(-1,30)$ and $Q=(7,34)$ be two points on $E(\Q)$. The point $Q$ has order $4$. Let $\OO=\Z$ and then we want to study the sequence $\{B_n(P,Q)\}_{n\in \Z}$. One can easily compute some terms.
	\begin{center}
		\begin{tabular}{|c| c|} 
			\hline
			$n$  & $B_n(P,Q)$\\
			\hline
			$0$& $(1)$\\
			\hline
			$1$& $(2)^2$\\
			\hline
			$2$& $(19)^2$\\ \hline
			$3$& $(6991)^2$\\ \hline
			$4$&$(12338681)^2$\\ \hline
			$5$& $(2)^2\cdot(4890590069)^2$\\
			\hline
		\end{tabular}
	\end{center}
	Let $\p=19$. This is a good prime. Then, one can compute that $\Ann_\p(P)=(8)$. Let $n=8q+r$ with $0\leq r\leq 7$. Hence,
	\[
	nP+Q\equiv 8qP+rP+Q\equiv rP+Q\mod {19}.
	\] 
	One can easily show that, for $0\leq r\leq 7$, we have $rP+Q\equiv O \mod{19}$ if and only if $r=2$. Therefore, $B_k(P,Q)$ is divisible by $19$ if and only if $k=8q+2$.
	
	Since $Q$ has order $4$, we have $\s=(4)$. Hence, in this case,
	\[
	I_\p(P)=\frac{\Ann_\p(P)}{\s}=\frac{(8)}{(4)}=(2).
	\]
	Fix $\alpha\in \OO=\Z$. Then,
	\[
	J_\p(P)=\frac{(\alpha)}{(2)}.
	\]
	In Lemma \ref{S} we will show that, if $\p$ divides $B_\alpha$, then $J_\p(P)$ is integral. In the case $\p=(19)$, $\p$ divides $B_\alpha$ for $\alpha=2+8q$ with $q\in \Z$, as we said before. Therefore, in this case,
	\[
	J_\p(P)=\frac{(\alpha)}{(2)}=\frac{(2+8q)}{(2)}=(1+4q)
	\]
	is an integral ideal.
\end{example}
\begin{lemma}\label{Ipmid}
	Let $\p$ be a prime divisor of $B_\alpha$ such that $\nu_\p\notin W$.
	Then, \[\Ann_{\p}(P)\mid (\alpha)\s.\]
\end{lemma}
\begin{proof}
	For every $s\in \s$,
	\[
	s\alpha P\equiv s\alpha P+sQ\equiv s(\alpha P+Q)\equiv O\mod{\p}.
	\]
	Therefore, $(\alpha)\s\subseteq \Ann_{\p}(P)$. Since $\OO$ is a Dedekind domain, this concludes the proof.
\end{proof}	  
\begin{lemma}\label{primdiv}
	Let $\p$ be a non-primitive divisor of $B_\alpha$ such that $\nu_\p\notin W$. Then, \[\Ann_{\p}(P)\neq (\alpha)\s.\]
\end{lemma}
\begin{proof}
	Assume that $\s=\OO$. Then, $Q=O$ and $\Ann_{\p}(P)\mid (\alpha)$ for the previous lemma. Since $\p$ is a non-primitive divisor, we have that there exists $\beta\in \Ann_{\p}(P)$ with $0<\Norm{\beta}<\Norm{\alpha}$. Every non-zero element of $(\alpha)$ has degree at least $\Norm{\alpha}$. Hence, $\beta\notin (\alpha)$. So, $(\alpha)\s=(\alpha)\subsetneq \Ann_{\p}(P)$. 
	
	Assume now $\s\neq \OO$. Observe that $\Ann_{\p}(Q)=\s\neq \OO$, since $\nu_\p\notin W$. So $Q$ is not the identity modulo $\p$. Recall that $\p$ divides $B_\alpha$ if and only if $\alpha P+Q\equiv O\mod \p$. 
	
	Suppose that $\Ann_{\p}(P)=(\alpha)\s$. We will prove that this assumption is absurd. Observe that if $\p$ divides $B_\beta$, then $\beta$ can be written in the form $\beta=\alpha+\alpha s$ for some $s\in \s$. Indeed,
	\[
	(\beta-\alpha)P\equiv (\beta P+Q)-(\alpha P+Q)\equiv O-O\equiv O \mod \p
	\]
	and then \[(\beta-\alpha)\in \Ann_{\p}(P)= (\alpha)\s.\] Let $\beta\in \OO$ be such that $\p$ divides $B_\beta$ and $0\leq \Norm{\beta}<\Norm{\alpha}$. Such a $\beta$ exists since $\p$ is a non-primitive divisor. If $\beta=0$, then
	\[
	Q\equiv \beta P+Q\equiv O\mod{\p},
	\]that is absurd since $Q$ is not the identity modulo $\p$. Hence, $\beta\neq0$. Therefore, for some $s\in \s$, \[\Norm{\beta}=\Norm{\alpha+\alpha s}=\Norm{\alpha}\Norm{1+s}\geq \Norm{\alpha}.\] This is absurd since we assumed $\Norm{\beta}<\Norm{\alpha}$. So, the assumption that $\Ann_{\p}(P)=(\alpha)\s$ is absurd and then $\Ann_{\p}(P)\neq(\alpha)\s$.
\end{proof}
\begin{lemma}\label{S}
	Let $\p$ be a prime divisor of $B_\alpha$ such that $\nu_\p\notin W$. Then $J_\p(P)$ is an integral $\OO$-ideal so that $(J_\p(P),\s)=1$, where $J_\p(P)$ is defined in Definition \ref{IJK}.
\end{lemma}
\begin{proof}
	We know that $\Ann_{\p}(P)\mid (\alpha)\s$ from Lemma \ref{Ipmid}, hence $J_\p(P)$ is an integral ideal. Consider the integral ideal $\s/(\s,J_\p(P))$ and take $s'\in \s/(\s,J_\p(P))$. Observe that \[s'\alpha\in \frac{\s(\alpha)}{(\s,J_\p(P))}=\left(\frac{J_\p(P)}{(\s,J_\p(P))}\right)\Ann_{\p}(P)\subseteq \Ann_{\p}(P).\] Hence, $s'\alpha P\equiv O\mod{\p}$ and, since $\p\mid B_\alpha$,
	\[
	s'Q\equiv s'Q+s'\alpha P\equiv s'(\alpha P+Q) \equiv O \mod{\p}.
	\]
	Therefore, $s'\in \s$ since $s'Q\equiv O \mod{\p}$ only for $s'\in \s$. Here we are using again that $\nu_\p\notin W$. Then, for any $s'\in \s/(\s,J_\p(P))$, we have $s'\in \s$. So,
	\[
	\frac{\s}{(\s,J_\p(P))}=\s.
	\]
	Finally, 
	\[
	(\s,J_\p(P))=1.
	\]
\end{proof}
Observe that, if $\p$ is a non-primitive divisor of $B_\alpha$, then $J_\p(P)$ is a proper ideal. This follows from Lemma \ref{primdiv}.
\begin{corollary}\label{defK}
	Let $\p$ be a prime divisor of $B_\alpha$ such that $\nu_\p\notin W$. Then $I_\p(P)$ is an integral $\OO$-ideal, where $I_\p(P)$ is defined in Definition \ref{IJK}.
\end{corollary}
\begin{proof}
	Using the definition and the previous lemma, 
	\[
	\frac{J_\p(P)\Ann_{\p}(P)}{\s}=(\alpha)
	\]
	and $(J_\p(P),\s)=1$. Therefore, $\s$ divides $\Ann_{\p}(P)$. So, \[I_\p(P)=\frac{\Ann_\p(P)}{\s}\] is integral.
\end{proof}
Observe that $I_\p(P)J_\p(P)=(\alpha)$. Moreover, if $\p$ is a non-primitive divisor of $B_\alpha$, then $I_\p(P)\neq (\alpha)$. This follows from Lemma \ref{primdiv}.
\begin{definition}\label{defI}
	Let $\I$ be the set of integral $\OO$-ideals defined as
	\[
	\I\coloneqq\{I\subseteq \OO \text{  s.t. }I\mid (\alpha) \text{ and }((\alpha)I^{-1},\s)=1 \}.
	\]
\end{definition}
Let $\p$ be a prime a divisor of $B_{\alpha}$ such that $\nu_\p\notin W$. Recall that, by definition, $I_\p(P)J_\p(P)=(\alpha)$. Observe that $I_\p(P)\in \I$. This follows from Lemma \ref{S} and Corollary \ref{defK}. We want to define a divisibility sequence $\B_I$ indexed by $I\in \I$.
\begin{definition}\label{defJ}
	Let $\J$ be the the set of integral $\OO$-ideals defined as
	\[
	\J\coloneqq\{J\subseteq \OO \text{  s.t. }J\mid (\alpha) \text{ and }(J,\s)=1 \}.
	\]
\end{definition}
Observe that the map $\psi:\I\to \J$, defined as
\[
\psi(I)=\frac{(\alpha)}{I},
\]
is well-defined and it is a bijection. This follows easily from Definition \ref{defI} and \ref{defJ}.

We defined the sets $\I$ and $\J$ since, for every good prime $\p$, we have $I_\p(P)\in \I$ and $J_\p(P)\in \J$.
\subsection{The sequence \texorpdfstring{$\B_I$}{}}
Let $\p$ be a prime divisor of $B_\alpha$ such that $\nu_\p\notin W$. Then, thanks to Lemma \ref{S}, $J_\p(P)\in \J$ and, thanks to Corollary \ref{defK}, $I_\p(P)\in \I$. We would like to define a sequence of integral $\OO_K$-ideals $\B_I$, indexed by the ideals $I\in \I$, with the following properties:
\begin{itemize}
	\item $B_\alpha=\B_{(\alpha)}$;
	\item if $\p$ is a divisor of $B_\alpha$, then $\p$ divides $\B_{I_\p(P)}$;
	\item if $\nu_\p(\B_I)>0$ for $I\in \I$, then $\nu(\B_{(\alpha)})$ is close to $\nu(\B_I)$.
\end{itemize}
If we find such a sequence, then we can replicate the strategy of the proof of Silverman and Streng, as we showed in Section \ref{structure}.

This sequence will be defined in Definition \ref{BB}. In order to do this, we need some preliminary results.
\begin{lemma}
	There exists $q\in \OO$ such that $q\in J$ for every $J\in \J$ and such that
	\[
	q\equiv 1\mod{\s}.
	\]
\end{lemma}
\begin{proof}
	Let $J_1$ be the unique minimal ideal of $\J$ when the ideals of $\J$ are ordered by inclusion. This ideal exists since $\J$ is closed under intersection.
	Considering that $(J_1,\s)=1$, we have that there exists $q\in J_1$ so that
	\[
	q\equiv 1\mod{\s}.
	\]
	For every $J\in \J$, we have $J_1\subseteq J$. Hence, $q\in J$.
\end{proof}
\begin{definition}\label{defq}
	From now on, we fix $q\in \OO$ that verifies the hypotheses of the previous lemma.
\end{definition}
\begin{definition}\label{Q'}
	Let $R\in E(\overline{\Q})$ be so that $\alpha R=Q$. We denote $R$ with $Q/\alpha$. Define
	\[
	Q'=q\cdot\frac{Q}{\alpha}\in E(\overline{\Q}),
	\]
	where $q$ is defined in the previous definition.
	The point $Q'$ depends on the (non-unique) choice of a point $Q/\alpha$ and it is a torsion point of $E(\overline{\Q})$. From now on, the point $Q'$ is fixed.
\end{definition}
Now, we do a short recap. We are studying the sequence $\{B_\beta\}_{\beta\in \OO}=\{B_\beta(P,Q)\}_{\beta\in \OO}$. We fixed $\alpha\in \OO$. Then, we defined the set of $\OO$-ideals $\I$ and $\J$, the element $q\in \OO$ and the point $Q'$. Everything here depends on $\alpha$. Moreover, if we fix a prime $\p$, we can define the set of fractional ideals $I_\p(P)$, $J_\p(P)$ and $\Ann_\p(P)$. These sets depend from $\alpha$, $\p$ and $P$.
As we said on the previous page, we want to define a sequence $\B_I$ of integral $\OO_K$-ideals for $I\in \I$. The term $\B_I$ will represent the denominator of $x(\varphi_I(P+Q'))$. 
In the last pages, we introduced plenty of notation. Hence, in order to help the reader to follow the proof of Theorem \ref{Thm1}, we give a summary table with the notation that we are using.
\begin{center}
\begin{tabular}{| c| }
 \hline $B_\alpha$ is the denominator of $x(\alpha P+Q) $\\ 
 \hline
 $\s$ is the order of $Q$ \\
 \hline  
 $\Ann_\p(R)=\{i\in \OO\mid iR\equiv O\mod \p\}$\\
 \hline
 $\varphi_I$ is an isogeny from $E$ to $E_{\overline{I}}$ such that $\Ker(\varphi_I)=I$\\\hline
 $W$ is the set of bad primes  \\
 \hline
 $I_\p(P)=\Ann_\p(P)/\s$\\
 \hline
 $J_\p(P)=(\alpha)/I_\p(P)$\\
 \hline
 $\I=\{I\subseteq \OO \text{  s.t. }I\mid (\alpha) \text{ and }((\alpha)I^{-1},\s)=1 \}$\\
 \hline
 $\J=\{J\subseteq \OO \text{  s.t. }J\mid (\alpha) \text{ and }(J,\s)=1 \}$\\\hline
 $q$ is an element of $\OO$ such that $q\in J$ for all $J\in \J$ and $q\equiv 1\mod \s$\\
 \hline
 $Q'=(qQ)/\alpha \in E(\overline{\Q})$\\\hline
 $\B_I$ is the denominator of $x(\varphi_I(P+Q'))$
 \\
 \hline
\end{tabular}
\end{center}
In order to prove that $\B_I$ is an $\OO_K$-ideal, we need two preliminary results.
\begin{lemma}\label{I1I2}
	Let $I\in \I$. Then, for every $i \in I$, the point $iP+iQ'\in E(K)$. Moreover, $iQ'$ is a multiple of $Q$. Finally, $\alpha Q'=Q$. 
\end{lemma}
\begin{proof}
	Let $J=(\alpha)I^{-1}$. This is an integral ideal considering that $I\in \I$. Observe that $J\in \J$ and hence $q\in J$. Since $iq\in IJ=(\alpha)$, we have $c\coloneqq iq/\alpha\in \OO$. So, \[iQ'=\frac{iq}{\alpha}Q=cQ\] and then \[iP+iQ'=iP+cQ\in E(K).\] Finally, \[\alpha Q'=\frac{q\alpha Q}{\alpha}=qQ=Q.\] The last equality follows from $q\equiv 1 \mod{\s}$.
\end{proof}
Recall that the map $\varphi_I$, as defined in the Section \ref{prel}, is an isogeny from $E$ to $E_{\overline{I}}$ with kernel the $I$-torsion points.
\begin{lemma}
	Let $I\in \I$. Then, $\varphi_I(P+Q')\in E_{\overline{I}}(K)$.
\end{lemma}
\begin{proof}
	Let $\sigma\in \Gal(\overline{\Q}/K)$ and then we want to show that \[\varphi_I(P+Q')=\varphi_I(P+Q')^\sigma.\] We know \[i((P+Q')^\sigma-(P+Q'))=(iP+iQ')^\sigma-(iP+iQ')=O\] for every $i\in I$ since $i(P+Q')\in E(K)$ thanks to Lemma \ref{I1I2}. So, \[(P+Q')^\sigma-(P+Q')\in E[I].\] Since the kernel of the map $\varphi_I$ is composed of the $I$-torsion points and $(P+Q')^\sigma-(P+Q')$ is an $I$-torsion point, we have that the image is the identity of $E_{\overline{I}}$. Then, \[\varphi_I(P+Q')^\sigma=\varphi_I(P+Q').\] This happens for every $\sigma\in \Gal(\overline{\Q}/K)$, so we have $\varphi_I(P+Q')\in E_{\overline{I}}(K)$.
\end{proof}
Recall that $\I$ is defined in Definition \ref{defI} and $Q'$ is defined in Definition \ref{Q'}.
\begin{definition}\label{BB}
	Let $I\in \I$.
	Thanks to the previous lemma, we know that $x(\varphi_I(P+Q'))$ is defined over $K$. So, we can define the $\OO_K$-ideals \[(x(\varphi_I(P+Q')))\OO_K=\frac{\A_I}{\B_I},\]
	where $\A_I$ and $\B_I$ are two coprime integral $\OO_K$-ideals.
\end{definition}
This definition is equivalent to the definition of $\tilde{B}_I$ given by Streng in \cite[Section 3]{streng}, where it is studied the case when $Q=O$.

Observe that, since $\alpha(P+Q')=\alpha P+Q$, we have \[B_\alpha=\B_{(\alpha)}.\]

This construction depends on $\alpha$, and, in general, $\B_{(\beta)}\neq B_\beta$. Moreover, $\B_I$ is defined only for the ideals $I\in \I$. Anyway, this sequence has two very useful properties, that we are going to prove later: it is a divisibility sequence and, if $\p$ is a non-primitive divisor of $B_\alpha$, then there exists $I\neq (\alpha)$ so that $\nu_\p(\B_I)>0$.

The goal of Section \ref{2prop} is to prove these two properties. In particular, we will show that, if $\p$ is a non-primitive divisor of $B_\alpha$, then we can bound $\nu_\p(\B_{(\alpha)})$ using $\nu_\p(\B_I)$ for $I\in \I\setminus\{(\alpha)\}$. This will be the main ingredient of the proof of Theorem \ref{Thm1}.

First, we return to the previous example.
\begin{example}
	Let $E$ be the rational elliptic curve defined by the equation $y^2=x^3-11x+890$. Let $P=(-1,30)$ and $Q=(7,34)$ be two points on the curve. We study $\{B_\alpha(P,Q)\}_{\alpha\in \Z}$. Fix $\alpha=8k+2$ for $k\in \Z$. We will focus on the $\alpha$ in this form, the generic case is similar. Since $\s=(4)$, we have
	\begin{align*}
		\I&=\left\{(d)\text{  such that } d \text{  divides  } 8k+2 \text{  and  } \left(\frac{8k+2}{d},4\right)=1\right \}
		\\&=\{(2d)\text{  such that } d \text{  divides  } 4k+1\}
	\end{align*}
	and 
	\[
	\J=\{(d)\text{  such that } d\text{  divides  } 4k+1\}.
	\]
	Take $q=4k+1$ and then $q \equiv 1 \mod{\s}$. Moreover, for every $J\in \J$, $q\in J$. Hence, we can define
	\[
	Q'=\frac{qQ}{\alpha}=\frac{Q}{2}.
	\]
	If $I=(2d)\in \I$, then $\varphi_I$ is the multiplication by $2d$. Therefore
	\[
	\varphi_I(Q')=2d\frac Q2=dQ
	\]
	and $\B_I$ is the denominator of $x(2dP+dQ)$.
	
	Since $I\in \I$, we have that $d$ is odd. If $d\equiv 1\mod{4}$, then $dQ=Q$ and $B_{2d}(P,Q)=\B_{(2d)}$. Otherwise, $d\equiv 3\mod {4}$ and $dQ=3Q$. So, $\B_{(2d)}=B_{2d}(P,3Q)$ and $B_{2d}(P,Q)\neq\B_{(2d)}$. For example, the denominator of $6P+Q$ is \[B_6(P,Q)=(43)^2\cdot(59)^2\cdot(3421265013773)^2\] and the denominator of $6P+3Q$ is \[\B_6=(19)^2\cdot (727)^2\cdot(102625619)^2\cdot(4877)^2.\] We will show later that, if $\p$ is a non-primitive divisor of $B_{\alpha}$, then there exists $I\in \I\setminus\{(\alpha)\}$ so that $\p$ divides $\B_{I}$. Take $\p=(19)$ and $\alpha=10$. The prime $\p$ is a non-primitive divisor of $B_{10}$ since it divides $B_2$. So, we expect that $\p$ divides $\B_{2d}$ for $d$ a proper divisor of $5=\alpha/2$. This is true since $\B_{2\cdot 1}=(19)^2$.
\end{example}
\subsection{Properties of the sequence \texorpdfstring{$\B_I$}{}}\label{2prop}
As we said, the goal of this section is to show how to bound $\nu(\B_{(\alpha)})$ using $\nu(\B_I)$. 

The next lemmas are necessary in order to prove the most important result of the next sections: Proposition \ref{nuW}.

\begin{lemma}\label{2nu}
	Let $I_1, I_2 \in \I$ with $I_1\mid I_2$ and $\nu\notin W$. Assume $\nu(\B_{I_1})>0$. Therefore, 
	\[
	\nu(\B_{I_2})=\nu(\B_{I_1})+2\nu(I_2)-2\nu(I_1).
	\] 
\end{lemma}
\begin{proof}
	We apply Lemma \ref{315}, using that $\nu\notin W$.
\end{proof}
\begin{definition}\label{defIp}
	Let $\p$ be so that $\nu_\p\notin W$. Define the set $\I_\p$ of integral $\OO_K$-ideals as \[\I_\p\coloneqq\{I\in \I\text{ such that } \nu_\p(\B_I)>0 \}\subseteq \I.\]
\end{definition}
We want to describe the set $\I_\p$.
In the next lemmas, we will prove that, if $\p$ is a divisor of $B_\alpha$, then $I_\p(P)$ is the largest ideal of $\I_\p$ when the elements of $\I_\p$ are ordered by inclusion. This will be proved in Lemma \ref{maxideal}. First of all, with the next lemmas, we show that $\nu(\B_I)>0$ for $I\subseteq I_\p(P)$ with $I\in \I$ and $\p$ a divisor of $B_\alpha$.

Recall that, if $R\in E(K)$, then $\nu_\p(x(R))<0$ if and only if $R\equiv O\mod{\p}$.
\begin{lemma}\label{I}
	Let $\p$ be a prime divisor of $B_\alpha$ such that $\nu_\p\notin W$. Let $I\subseteq I_{\p}(P)$ be an integral $\OO$-ideal in $\I$. Then, for every $i\in I$, \[\nu_\p(x(iP+iQ'))<0.\] 
\end{lemma}
\begin{proof}
	Let $J=(\alpha)I^{-1}$. Note that $J$ is coprime to $\s$. Therefore, $J\in \J$ and $q\in J$.
	Fix $i\in I$. We know $IJ=(\alpha)$ and then $iq=c\alpha$ for some $c\in \OO$. Recall that $q\equiv 1 \mod{\s}$ and observe that \[(q-1)i\in \s I\subseteq \s I_{\p}(P)= \Ann_{\p}(P).\] Then, $(1-q)iP\equiv O\mod{\p}$ and
	\[
	O\equiv c(\alpha P+Q)+(1-q)iP\equiv iqP+cQ+(1-q)iP\equiv iP+c Q \mod \p.
	\]	
	We conclude observing that
	\[
	iQ'=\frac{iq Q}{\alpha}=cQ
	\]
	and then
	\[
	iP+iQ'\equiv iP+cQ\equiv O\mod{\p}.
	\]
\end{proof}
\begin{lemma}
	Let $I\in \I$ and $\nu\notin W$. 	Let
		\[
		v_I\coloneqq\min\{-\nu(x(i(P+Q')))\mid i\in I\}.
		\]
		Then, \[\nu(\B_I)=v_I.\]
\end{lemma}
\begin{proof}
	This is proved in \cite[Lemma 3.17]{streng}. In order to apply \cite[Lemma 3.17]{streng}, we need that $I$ is an invertible ideal. This is the case since in the Dedekind domain $\OO$ every integral ideal is invertible.
\end{proof}
\begin{lemma}\label{nu(D)}
	Let $I\in \I$ and $\nu\notin W$. Then, $\nu(\B_I)>0$ if and only if $\nu_\p(x(iP+iQ'))<0$ for every $i\in I$.
\end{lemma}
\begin{proof}
	Observe that $v_I>0$ if and only if $\nu_\p(x(iP+iQ'))<0$ for every $i\in I$. We conclude using the previous lemma.
\end{proof}
\begin{corollary}\label{cornu0}
	Let $\p$ be a prime divisor of $B_\alpha$ such that $\nu_\p\notin W$. Let $I\subseteq I_{\p}(P)$ be an integral $\OO$-ideal in $\I$. Then,
	\[
	\nu_\p(\B_I)>0.
	\]
\end{corollary}
\begin{proof}
	This follows easily from the last three lemmas.
\end{proof}
Thanks to the last corollary, we know that, if $I\subseteq I_\p(P)$, then $I\in \I_\p(P)$.
Now, we show that $I_\p(P)$ is the largest ideal of $\I_\p$ when the elements of $\I_\p$ are ordered by inclusion.
\begin{lemma}\label{max}
	Let $\p$ be a prime $\OO$-ideal such that $\nu_\p\notin W$. If $\I_\p$ is not empty, then it has a unique maximal element when the elements of $\I_\p$ are ordered by inclusion.
\end{lemma}
\begin{proof}
	We just need to prove that if $I_1$ and $I_2$ belong to $\I_\p$, then $I_3\coloneqq(I_1+I_2)$ belongs to $\I_\p$. If $i_3\in I_3$, then $i_3=i_1+i_2$ for $i_1\in I_1$ and $i_2\in I_2$. So,
	\[
	i_3(P+Q')=i_1(P+Q')+i_2(P+Q').
	\]
	We know $\nu_\p(x(i_1(P+Q'))<0$ and $\nu_\p(x(i_2(P+Q'))<0$ thanks to Lemma \ref{nu(D)} and so $\nu_\p(x(i_3(P+Q'))<0$ for every $i_3\in I_3$. Then, using again Lemma \ref{nu(D)}, $\nu(\B_{I_3})>0$. Moreover, $((\alpha)I_3^{-1},\s)=1$ since the primes that divide $I_3$ are the primes that divide both $I_1$ and $I_2$. Therefore, $I_3\in \I_\p$.
\end{proof}
\begin{lemma}\label{maxideal}
	Let $\p$ be a prime divisor of $B_\alpha$ such that $\nu_\p\notin W$. Then $I_\p(P)$ is the largest ideal of $\I_\p$ when the elements of $\I_\p$ are ordered by inclusion.
\end{lemma}
\begin{remark}
	This lemma allows us to understand when $\nu(\B_I)$ is strictly positive. Indeed, we have that $\nu(\B_I)>0$ if and only if $I$ is a multiple of $I_\p(P)$.
\end{remark}
\begin{proof}
	Since $\nu(\B_{(\alpha)})=\nu(B_\alpha)>0$, we have that $I_\p(P)$ is not empty. Recall that, by definition, $I_{\p}(P)J_{\p}(P)=(\alpha)$. We know $(J_{\p}(P),\s)=1$ thanks to Lemma \ref{S}. Moreover $I_{\p}(P)\in \I_{\p}$ thanks to Corollary \ref{cornu0} and we denote with $I_1$ the unique largest ideal of $\I_\p$. So, $I_{\p}(P)\subseteq I_1$. Therefore, $I_1I_2=I_{\p}(P)$ with $I_2$ an $\OO$-ideal. We want to show that $I_1=I_{\p}(P)$. Let $J_1$ be so that $I_1J_1=(\alpha)$ with $J_1 \in \J$ since $I_1\in \I_\p\subseteq \I$.
	Let $i_1\in I_1$ and then $i_1q=c\alpha$ since $I_1J_1=(\alpha)$ and $q\in J_1$. Therefore,
	\[
	i_1(P+Q')=i_1qP+i_1(1-q)P+cQ=c(\alpha P+Q)+i_1(1-q)P.
	\]
	We know that $i_1(P+Q')$ reduces to the identity modulo $\p$ since $I_1\in\I_\p$, and so \[i_1(1-q)P\equiv O \mod{\p}.\] Thus, $i_1(1-q)\in \Ann_\p(P)= I_{\p}(P)\s\subseteq I_{\p}(P)$ for every $i_1\in I_1$. Observe that, by definition, \[J_1=\frac{(\alpha)}{I_1}=\frac{(\alpha)I_2}{I_\p(P)}=I_2J_{\p}(P).\] So, $I_2\in \J$. Hence $q\in I_2$ and $qi_1\in I_2I_1=I_{\p}(P)$. Therefore, for every $i_1\in I_1$,
	\[
	i_1=i_1q+(1-q)i_1\in I_{\p}(P)
	\]
	since both summands belong to $I_{\p}(P)$.
	In conclusion, for every $i_1\in I_1$, we have $i_1\in I_\p(P)$. So, $I_{\p}(P)=I_1$.
\end{proof}
\subsection{Bound on the valuation of \texorpdfstring{$\B_\alpha$}{}}
Now, we have the sequence of $\OO_K$-ideals $\B_I$ for $I\in \I$, we take a good prime $\p$ and we want to bound $\nu(\B_{(\alpha)})$.

Let $\p$ be a non-primitive divisor of $\B_{(\alpha)}$ such that $\nu=\nu_\p\notin W$. Then, by Lemma \ref{primdiv}, $I\coloneqq I_\p(P)$ is a proper divisor of $\alpha$ that belongs to $\I$. Thanks to Lemma \ref{maxideal}, we know that $\nu(\B_I)>0$. In order to prove Theorem \ref{Thm1} we need to bound (in some sense) $\nu(B_\alpha)=\nu(\B_{(\alpha)})$. From Lemma \ref{2nu}, we know that $\nu(\B_I)$ is close to $\nu(\B_{(\alpha)})$. So, in order to bound $\nu(\B_{(\alpha)})$, we just need to bound $\nu(\B_I)$, for every $I\in \I$ not equal to $(\alpha)$.
\begin{definition}\label{mob}
	Given $I\neq 0$ an integral $\OO$-ideal, write $I=\prod \q_i^{a_i}$ the unique factorization of $I$ as product of primes. If one of the $a_i>1$, put $\mu(I)=0$. Otherwise, put $\mu(I)=(-1)^k$, where $k$ is the number of prime divisors of $I$. This is the so-called M\"obius function.
\end{definition}
Let $M_k^0$ be the set of finite valuations of $K$.
\begin{lemma}\label{propmu}
If $I\neq \OO$ is a non-zero proper $\OO$-ideal, then
\[
\sum_{J\mid I}\mu(J)=0.
\]
Moreover, given $\nu\in M_K^0$, we have
\[
\sum_{J\mid I}\mu(J)\nu(J)=\begin{cases}
-1 \mbox{  if  } I \mbox{ is a power of } \p,\\
0 \mbox{ otherwise, }
\end{cases}
\]
where $\p$ is the prime associated with $\nu$.
\end{lemma}
\begin{proof}
This follows easily from an inclusion-exclusion argument. For some details, see \cite[Theorem 263]{hardywright}.
\end{proof}
\begin{lemma}
	Let $\p$ be a prime $\OO$-ideal such that $\p$ is not a primitive divisor of $B_{\alpha}$ and such that $\nu=\nu_\p\notin W$. Then,
	\[
	\nu\left(\B_{(\alpha)}\right)\leq -\sum_{I\in \I\setminus (\alpha)}\mu((\alpha)I^{-1})\nu(\B_{I})+2\nu(\alpha).
	\]
\end{lemma}
\begin{remark}
	The hypothesis $\p$ non-primitive divisor is necessary and it is the key point of the proof of Theorem \ref{Thm1}. We show an example where is necessary. If $Q=O$ and $\OO=\Z$, then $\B_{(k)}=B_k(P,O)$. If $\p$ is a primitive divisor of $B_n$ and $\nu_\p(n)=0$, then the inequality does not hold since $\nu_\p(\B_{(k)})=\nu_\p(B_k(P,O))=0$ for $k\mid n$ and $k\neq n$. Observe that $\Z$ is a PID, so every ideal in $\I\setminus (n)$ is in the form $(k)$ with $k$ a proper divisor of $n$.
\end{remark}
\begin{proof}
	If $\nu(\B_{(\alpha)})=0$, then $\nu(\B_I)=0$ for every $I\in \I$ thanks to Lemma \ref{2nu}. In this case, the lemma holds. So, we assume $\nu(\B_{(\alpha)})>0$. Thanks to Lemma \ref{I}, we know that $I_\p(P)$ is in the set $\I_\p$ and it is not $(\alpha)$ using Lemma \ref{primdiv}, since $\p$ is a non-primitive divisor. Moreover, $I_\p(P)\neq (\alpha)$ is the largest ideal of $\I_\p$ when the elements of $\I_\p$ are ordered by inclusion. So, $\I_\p=\{I\text{ such that } I_\p(P)\mid I \text{ and } I\mid(\alpha)\}$. So, using Lemma \ref{2nu},
	\begin{align}\label{line}
		\sum_{I\in \I_\p}\mu\left((\alpha)I^{-1}\right)\nu(\B_{I})&=\sum_{I\in \I_\p}\mu\left((\alpha)I^{-1}\right)\left(\nu(\B_{I_\p(P)})+2\nu(II_{\p}(P)^{-1})\right)\nonumber\\&=\sum_{I\in \I_\p}\mu((\alpha)I^{-1})\nu(\B_{I_\p(P)})+2\mu((\alpha)I^{-1})\nu(II_\p(P)^{-1})\nonumber\\&=\nu(\B_{I_\p(P)})\sum_{K\mid (\alpha)I_\p(P)^{-1}}\mu(K)+2\nu\left(\frac{(\alpha)}{I_\p(P)}\right)\sum_{K\mid (\alpha)I_\p(P)^{-1}}\mu(K)\nu(K^{-1})\\&\leq \nu(\B_{I_\p(P)})\cdot 0+2\nu\left(\alpha I_\p(P)^{-1}\right)\nonumber\\&\leq 2\nu(\alpha)\nonumber.
	\end{align}
	In the last inequalities, we used Lemma \ref{propmu}. In Equation (\ref{line}), we put $K=(\alpha)/I$. Therefore,
	\[
	\sum_{I\in \I_\p}\mu\left((\alpha)I^{-1}\right)\nu(\B_{I})\leq2\nu(\alpha).
	\]
	So,
	\[
	\nu(\B_{(\alpha)})\leq -\sum_{\substack{I\in \I_\p\\I\neq (\alpha)}}\mu\left((\alpha)I^{-1}\right)\nu(\B_{I})+2\nu(\alpha).
	\]
	Moreover, recalling that $\I_\p$ is the set of ideals such that $\nu(\B_I)>0$, we have
	\[\sum_{\substack{I\in \I_\p\\I\neq (\alpha)}}\mu\left((\alpha)I^{-1}\right)\nu(\B_{I})=\sum_{I\in \I\setminus (\alpha)}\mu\left((\alpha)I^{-1}\right)\nu(\B_I).
	\]
	Hence,
	\[
	\nu(\B_{(\alpha)})\leq -\sum_{I\in \I\setminus (\alpha)}\mu\left((\alpha)I^{-1}\right)\nu(\B_{I})+2\nu(\alpha).
	\]
\end{proof}
\begin{proposition}\label{nuW}
	Let $\p$ be a non-primitive divisor of $B_{\alpha}$ such that $\nu=\nu_\p\notin W$. Then,
	\[
	h_\nu(\alpha(P+Q'))\leq 2h_\nu(\alpha^{-1})+\sum_{I\in \I\setminus\{(\alpha)\}}-\mu\left((\alpha)I^{-1}\right) h_\nu(\varphi_I(P+Q')).
	\]
\end{proposition}
\begin{proof}
Let $\pi$ be a uniformizer of $K_\nu$. Then, for $x\in K_\nu$, we have $\abs{x}_\nu=\abs{\pi}_\nu^{\nu(x)}$. Hence, for every $x\in K$ with $\nu(x)\leq 0$,
	\[h_\nu(x)=\log \abs{x}_\nu=\nu(x)(\log \abs{\pi}_\nu)=-C_\nu\nu(x)\] where the positive constant $C_\nu$ is equal to $-\log \abs{\pi}_\nu$. So, using the previous lemma,
	\begin{align*}
		h_\nu(\alpha (P+Q'))&=C_\nu\nu(\B_{(\alpha)})\\&\leq 2C_\nu\nu(\alpha)+C_\nu\sum_{I\in \I\setminus\{(\alpha)\}}-\mu\left((\alpha)I^{-1}\right) \nu(\B_I)\\&=2h_\nu(\alpha^{-1})+\sum_{I\in \I\setminus\{(\alpha)\}}-\mu\left((\alpha)I^{-1}\right) h_\nu(\varphi_I(P+Q')).
	\end{align*}
\end{proof}
\section{Study of the bad places}\label{secbadprime}
Before proceeding with the proof of Theorem \ref{Thm1}, we need to deal with the places in $W$, that we called the \textbf{bad places}. Recall that $W$ is defined in Definition \ref{defW}. We are going to show that, in order to compute $h(\varphi_I(P+Q'))$, the terms $h_\nu(\varphi_I(P+Q'))$ for $\nu \in W$ are negligible.

Recall that we fixed $P$ a non-torsion point and $Q$ a torsion point of $E(K)$. Moreover, we fix $\alpha\in \OO$, the point $Q'$ as defined in Definition \ref{Q'} and the set $\I$ as defined in Definition \ref{defI}. We assume that $\Norm{\alpha}> 1$. Since we want to prove that $B_\alpha$ has a primitive divisor for all but finitely many $\alpha$, this hypothesis does not affect the proof of the theorem.
\begin{lemma}\label{strengcor}
Fix an immersion $K\hookrightarrow \C$. There exists a constant $C_1$, depending only on $P$, such that, for all but finitely many integral $\OO$-ideals $I$, we have
\[
\abs{x(\varphi_I(P))}\leq C_1(\log \Norm{\varphi_I})^2.
\]
\end{lemma}
\begin{proof}
This is proved in \cite[Corollary 3.14]{streng}.
\end{proof}
\begin{lemma}
Given an elliptic curve $E'$ defined over $\C$ and $s\in \Z$, there exists a constant $C=C(E',s)$ so that, if $\abs{x(R)}>C$ for $R\in E'(\C)$, then \[\abs{x(sR)}>\abs{x(R)}/2s^2.\]
\end{lemma}
\begin{proof}
Consider the lattice $\Lambda$ such that $\C/\Lambda \cong E'(\C)$. The two groups are isomorphic via the function $\varphi$ as defined in \cite[Chapter VI]{arithmetic}. In particular $x(\varphi(z))=\wp(z)$ with
\[
\wp(z)=\frac{1}{z^2}+\sum_{\omega\in \Lambda\setminus \{0\}}\frac{1}{(z-\omega)^2}-\frac{1}{\omega^2}.
\]
Let $\Lambda_1$ be the closure of the set of complex numbers $z$ such that the element of $\Lambda$ closest to $z$ is $0$. Observe that, for every $R\in E'(\C)$, there exists $z\in \Lambda_1$ such that $\varphi(z)=R$. Let
\[
C_2=\max_{z\in \Lambda_1}\left\{\abs{\sum_{\omega\in \Lambda\setminus \{0\}}\frac{1}{(z-\omega)^2}-\frac{1}{\omega^2}}\right\}.
\]
This is well defined since $\Lambda_1$ is compact and it is finite since $\wp$ has pole only in the elements of $\Lambda$. Take \[r\leq ((2s^2+1)C_2)^{-\frac{1}{2}}\] and consider the ball $B(0,r)$ with radius $r$ and centre $0$. Substituting $r$ with a smaller value, we can assume that $B(0,sr)\subseteq \Lambda_1$.
Take $z \in B(0,r)$ and then $sz\in \Lambda_1$. So,
\begin{equation}\label{eqwp}
\abs{\wp(sz)}\geq \abs{(sz)^{-2}}-C_2\geq \frac{\abs{z^{-2}}+C_2}{2s^2}.
\end{equation}
The last inequality follows from the definition of $r$. Let $C=r^{-2}+C_2$. This constant depends only from $s$ and $\Lambda$. If $z\in \Lambda_1\setminus B(0,r)$, then \[\abs{\wp(z)}\leq r^{-2}+C_2=C.\] So, if $\abs{x(R)}>C$, then there exists $z\in B(0,r)$ such that $\varphi(z)=R$. Then, using (\ref{eqwp}),
\[
\abs{x(sR)}=\abs{\wp(sz)}\geq \frac{\abs{z^{-2}}+C_2}{2s^2}\geq \frac{\abs{x(R)}}{2s^2}.
\]
\end{proof}
\begin{lemma}
	There exists $C_3$, depending only on $P$, $Q$, and $E$ such that\[
	\abs{x(\varphi_I(P+Q'))}_\nu\leq C_3 \log(\Norm{\alpha})^2
	\]
	for every archimedean place $\nu$ of $K$ and for every $I\in \I$.
\end{lemma}
\begin{proof}
	Fix an immersion $K\hookrightarrow \C$ and let $\abs{\cdot}$ be the absolute value in $\C$ that extends $\abs{\cdot}_\nu$. Observe that, for every $s\in \s$ and $i\in I$, we have $siQ'=O$ since $iQ'$ is a multiple of $Q$ as is shown in Lemma \ref{I1I2}. So, given $s\in \s\cap \Z$, we have $s\varphi_I(Q')=O$.
	Thanks to Lemma \ref{strengcor}, for all but finitely many $\alpha$,
	\begin{equation}\label{loga2}
	\abs{x(s\varphi_I(P+Q'))}=\abs{x(s\varphi_I(P))}<C_4(\log \Norm{\alpha})^2
	\end{equation}
	with $C_4$ that depends only on $P$, $s$, and $E$.
	 Assume $C_4(\log \Norm{\alpha})^2>C(E_{\overline{I}},s)$ for every class $\overline{I}$, that happens for all but finitely many $\alpha$. The constant $C(E_{\overline{I}},s)$ is defined in the previous lemma. So, if \[\abs{x(\varphi_I(P+Q'))}>2s^2C_4(\log \Norm{\alpha})^2,\] then $\abs{x(s\varphi_I(P+Q'))}>C_4(\log \Norm{\alpha})^2$ for the previous lemma. This is absurd for (\ref{loga2}). Hence, \[\abs{x(\varphi_I(P+Q'))}\leq 2s^2C_4(\log \Norm{\alpha})^2\leq C_3(\log \Norm{\alpha})^2\] for all but finitely many $\alpha$. We conclude enlarging the constant $C_3$ for the finite remaining $\alpha$.
\end{proof}
\begin{lemma}\label{hat}
	For every $I\in \I$,
	\[
	2\hat{h}(\varphi_I(P+Q'))=\frac{1}{[K:\Q]}\sum_{\nu \notin W}h_\nu (\varphi_I(P+Q'))+O\left((\log\Norm{\alpha})^2\right).
	\]
\end{lemma}
The implied constant in the O-notation depends only from $E$, $P$, and $Q$.
\begin{proof}
	We start by observing that 
	\[
	\frac{1}{[K:\Q]}\sum_{\nu \in W}h_\nu (\varphi_I(P+Q'))=O\left((\log\Norm{\alpha})^2\right).
	\]
	The sum is on finitely many terms and so we need to check that \[h_\nu (\varphi_I(P+Q'))=O\left((\log\Norm{\alpha})^2\right)\] for every $\nu \in W$. 
	
	For the archimedean places, this follows from the previous lemma. 
	
	For the non-archimedean places, this follows from
	\[
	h_\nu (\varphi_I(P+Q'))\leq h_\nu(\alpha P+Q)\leq h_\nu(s\alpha P)\leq h_\nu(\Norm{s\alpha}P) =O(\log\Norm{\alpha})
	\]
	where the last equality follows from \cite[Lemma 3.7]{streng}. Then,
	\begin{align*}
	h(\varphi_I(P+Q'))&=\frac{1}{[K:\Q]}\sum_{\nu \notin W}h_\nu (\varphi_I(P+Q'))+\frac{1}{[K:\Q]}\sum_{\nu \in W}h_\nu (\varphi_I(P+Q'))\\&=\frac{1}{[K:\Q]}\sum_{\nu \notin W}h_\nu (\varphi_I(P+Q'))+O\left((\log\Norm{\alpha})^2\right).
	\end{align*}
	Observe now that the set of elliptic curves $E_{\overline{I}}(K)$ is finite and so it is possible to find a constant $C_5$ so that
	\[
	\abs{2\hat{h}(R)-h(R)}\leq C_5
	\]
	for every $R\in E_{\overline{I}}(K)$ and for every $I$. So,
	\begin{align*}
		2\hat{h}(\varphi_I(P+Q'))&=h(\varphi_I(P+Q'))+O(1)\\&=\frac{1}{[K:\Q]}\sum_{\nu \notin W}h_\nu (\varphi_I(P+Q'))+O\left((\log\Norm{\alpha})^2\right).
	\end{align*}
\end{proof}
\section{Proof of Theorem \ref{Thm1}}\label{secproof}

Now, we are ready to prove Theorem \ref{Thm1}. The two main ingredients are Proposition \ref{nuW} and Lemma \ref{hat}. With Lemma \ref{hat} we know that, in order to bound $h(\alpha P+Q)$, we can study only the good primes, and Proposition \ref{nuW} allows us to bound $h_{\nu}(\alpha P+Q)$ for $\nu\notin W$. Recall that $W$ is defined in Definition \ref{defW} and $\I$ is defined in Definition \ref{defI}.
\begin{proof}[Proof of Theorem \ref{Thm1} ]
	Suppose that $B_{\alpha}$ does not have a primitive divisor for $\Norm{\alpha}>1$. 
	Thanks to Proposition \ref{nuW},
	\[
	\sum_{\nu \notin W}h_\nu(\alpha (P+Q'))\leq\left(\sum_{\nu \notin W}\sum_{\substack{I\in \I\setminus\{(\alpha)\}}}-\mu\left(\frac{(\alpha)}{I}\right) h_\nu(\varphi_I(P+Q'))\right)+2[K:\Q]h(\alpha^{-1}).
	\]
	It is easy to show, using the definition, that $h(\alpha^{-1})=h(\alpha)=O(\log \Norm{\alpha})$. Thus, from Lemma \ref{hat},
	\begin{align*}
		2\hat{h}(\alpha (P+Q'))&\leq \frac{1}{[K:\Q]}\left(\sum_{\substack{I\in \I\setminus\{(\alpha)\}}}\sum_{\nu \notin W}-\mu\left(\frac{(\alpha)}{I}\right) h_\nu(\varphi_I(P+Q'))\right)+O\left(\log\Norm{\alpha}^2\right)\\&\leq\left(\sum_{I\in \I\setminus\{(\alpha)\}}-\mu((\alpha)I^{-1}) h(\varphi_I(P+Q'))\right)+O\left(\log\Norm{\alpha}^2\right)\\&\leq\sum_{I\in \I\setminus\{(\alpha)\}}-2\mu((\alpha)I^{-1}) \hat{h}(\varphi_I(P+Q'))+\sum_{I\in \I}O\left(\log\Norm{\alpha}^2\right)\\&\leq \left(\sum_{I\in \I\setminus\{(\alpha)\}}-2\mu((\alpha)I^{-1}) \hat{h}(\varphi_I(P+Q'))\right)+O\left(\Norm{\alpha}^\eps\right).
	\end{align*}
	Here we are using that $\I$ is a subset of the divisors of $(\alpha)$ and then has cardinality $O(\Norm{\alpha}^{\eps})$ for every $\eps>0$. Recall that the number of divisors of an ideal $I_1$ in $\OO$ is $O(\Norm{I_1}^\eps)$ for every $\eps>0$.
	We know $f_J(\varphi_I(P+Q'))=\alpha (P+Q')$ and $f_J$ has degree $\Norm{J}$, from the definitions of Section \ref{prel}.
	Therefore,
	\[
	\Norm{J}\hat{h}(\varphi_I(P+Q'))=\hat{h}(f_J(\varphi_I(P+Q')))=\hat{h}(\alpha (P+Q'))
	\]
	from Lemma \ref{canheight} and so
	\[
	\hat{h}(\alpha (P+Q'))\leq\sum_{\substack{1\neq J\mid (\alpha)\\(J,\s)=1}}-\mu(J)\frac{\hat{h}(\alpha (P+Q'))}{\Norm{J}}+O\left(\Norm{\alpha}^\eps\right).
	\]
	Recall that $Q'$ is a torsion point. Hence,\[\hat{h}(\alpha (P+Q'))=\hat{h}(\alpha P)=O(\Norm{\alpha}).\] Dividing by $\hat{h}(\alpha (P+Q'))$, we obtain
	\[
	1\leq\sum_{\substack{1\neq J\mid (\alpha)\\(J,\s)=1}}-\mu(J)\frac{1}{\Norm{J}}+O\left(\Norm{\alpha}^{-1+\eps}\right).
	\]
	So,
	\begin{equation}\label{eq:O}
		\sum_{\substack{J\mid (\alpha)\\(J,\s)=1}}\mu(J)\frac{1}{\Norm{J}}= O\left(\Norm{\alpha}^{-1+\eps}\right).
	\end{equation}
	Observe that here we used the hypothesis that $P$ is a non-torsion point and then $\hat{h}(P)>0$.
	Now, we conclude replicating the work in \cite[Proof of the main theorem, page 204]{streng}. Recall that $\Norm{J}$ is the degree of the isogeny $\varphi_{J}$. Using Definition \ref{mob} of the M\"obius function, we have
	\[
	\sum_{\substack{J\mid (\alpha)\\(J,\s)=1}}\mu(J)\frac{1}{\Norm{J}}=\prod_{\substack{\p\mid (\alpha)\\(\p,\s)=1}}\left(1-\frac{1}{\Norm{\p}}\right),
	\] 
	where the product runs over the prime $\OO$-ideals that divide $(\alpha)$.
	Observe that, for every prime divisor of $(\alpha)$, we have $\Norm{\p}\leq \Norm{\alpha}$. Hence,
	\[
	\sum_{\substack{J\mid (\alpha)\\(J,\s)=1}}\mu(J)\frac{1}{\Norm{J}}=\prod_{\substack{\p\mid (\alpha)\\(\p,\s)=1}}\left(1-\frac{1}{\Norm{\p}}\right)\geq \prod_{\Norm{\p}\leq \Norm{\alpha}}\left(1-\frac{1}{\Norm{\p}}\right).
	\]
	Given a prime $\p$, let $p$ be the rational prime over it. So, $\Norm{\p}\geq p$. There are at most two prime $\OO$-ideals over every rational prime $p$. Hence,
	\begin{equation}\label{eqmu}
	\sum_{\substack{J\mid (\alpha)\\(J,\s)=1}}\mu(J)\frac{1}{\Norm{J}}\geq \prod_{\Norm{\p}\leq \Norm{\alpha}}\left(1-\frac{1}{\Norm{\p}}\right)\geq \prod_{p\leq  \Norm{\alpha}}\left(1-\frac{1}{p}\right)^2
	\end{equation}
	where the last product runs over the rational primes.
	By Mertens' Theorem (see \cite[Theorem 429]{hardywright}), there exists an absolute constant $C_6$ so that
	\[
	\frac{C_6}{(\log\Norm{\alpha})}\leq \prod_{p\leq  \Norm{\alpha}}\left(1-\frac{1}{p}\right).
	\] 
	So, using (\ref{eqmu}),
	\[
	\sum_{\substack{J\mid (\alpha)\\(J,\s)=1}}\mu(J)\frac{1}{\Norm{J}}\geq \prod_{p\leq  \Norm{\alpha}}\left(1-\frac{1}{p}\right)^2\geq \frac{C_6^2}{(\log\Norm{\alpha})^2}.
	\]
	Hence, using (\ref{eq:O}),
	\[
	\frac{C_6^2}{(\log\Norm{\alpha})^2}\leq \sum_{\substack{J\mid (\alpha)\\(J,\s)=1}}\mu(J)\frac{1}{\Norm{J}}= O\left(\Norm{\alpha}^{-1+\eps}\right)
	\]
	and then
	\[
	\Norm{\alpha}^{1-2\eps}\leq C_7,
	\]
	where $C_7$ depends on $E,P,Q$ and $\eps$. Choosing $\eps$ small enough, the inequality holds only for finitely many $\alpha$. Hence, $B_\alpha$ does not have a primitive divisor only for $\Norm{\alpha}$ small. So, for all but finitely many $\alpha$, $B_\alpha$ has a primitive divisor.
\end{proof}
\section{Proof of Theorem \ref{Thm2}}\label{f=g}
Before proceeding with the proof of Theorem \ref{Thm2}, we need some preliminary lemmas. These lemmas could seem a bit unrelated to the topics of the paper, but they will be fundamental for the proof of Theorem \ref{Thm2}. In particular, Lemma \ref{K} will be necessary to prove Equation (\ref{eq1}) and Lemma \ref{C} will be necessary to prove Equation (\ref{eqC8}).

 Recall that $L=\OO\otimes_\Z\Q$ is a field of degree at most $2$. We fix the immersion $L\hookrightarrow \C$ as in Section \ref{prel}. Hence, we have an absolute value $\abs{\cdot}$ for the elements of $\OO$. If $\alpha \in \OO$, we have $\abs{\alpha}^2=\Norm{\alpha}$, where $\Norm{\alpha}$ is the degree of the endomorphism. This follows easily from the work in \cite[Section III.6]{arithmetic} and \cite[Section III.9]{arithmetic}.
\begin{lemma}
	Fix $\gamma\in \C$ so that $\gamma^2\in \Z^{<0}$. Let $z_1$ and $z_2\neq 0$ be two complex numbers with $2\abs{\gamma}\abs{z_2}<\abs{z_1}$. There exists $t\in\{1,-1,\gamma,-\gamma\}$ such that\[
	\abs{z_1-tz_2}< \abs{z_1}.
	\]
\end{lemma}
\begin{remark}
	The idea of the proof is easy from a geometric point of view. We have a vector $w$ in the complex plane and a set of four vectors $v_i$ that are orthogonal, with $\abs{v_i}\ll \abs{w}$. Then, one can observe that there is at least one of the $v_i$ such that $\abs{w-v_i}<\abs{w}$.
\end{remark}
\begin{proof}
	We begin with an easy observation. If $z_1'=xz_1$ and $z_2'=xz_2$ with $x\in \C^*$, then the lemma holds for $z_1$ and $z_2$ if and only if it holds for $z_1'$ and $z_2'$. Indeed, the inequalities $2\abs{\gamma}\abs{z_2}<\abs{z_1}$ and $\abs{z_1-tz_2}< \abs{z_1}$ hold for $z_1$ and $z_2$ if and only if they hold for $z_1'$ and $z_2'$. So, during the proof, we can multiply $z_1$ and $z_2$ by the same constant.

	We can write in a unique way $z_1=x_1+\gamma y_1$ and $z_2=x_2+\gamma y_2$ with $x_1,y_1,x_2,y_2\in \R$ since $\gamma$ is imaginary.
	Multiplying $z_1$ and $z_2$ by an appropriate element with absolute value $1$, we can assume $x_1=\abs{\gamma} y_1>0$. We can do this for the observation at the beginning of the proof. Choose $t\in\{1,-1,\gamma,-\gamma\}$ so that $tz_2=x_3+\gamma y_3$ with $x_3\geq 0$ and $y_3\geq 0$. For example, if $x_2\geq 0$ and $y_2\leq 0$, we put $t=\gamma$. So, $tz_2=\gamma^2 y_2+\gamma x_2$ and then $x_3=\gamma^2 y_2\geq 0$ and $y_3=x_2\geq 0$. The other cases are analogous. We will show that $\abs{z_1-tz_2}< \abs{z_1}$. By definition,
	\[
	\abs{z_1-tz_2}^2=(x_1-x_3)^2+\abs{\gamma}^2(y_1-y_3)^2.
	\]
	Recall that, by hypothesis, $ 2\abs{\gamma}\abs{z_2}<\abs{z_1}$. We have 
	\[
	0\leq x_3\leq \abs{t}\abs{z_2}\leq \abs{\gamma}\abs{z_2}< \frac{\abs{z_1}}{2}< x_1
	\] 
	where the last equality follows from
	\[
	\abs{z_1}^2=\abs{x_1}^2+\abs{\gamma y_1}^2=2\abs{x_1}^2.
	\]
	So, $0\leq  x_1-x_3\leq x_1$. Therefore, $(x_1-x_3)^2\leq x_1^2$. In the same way, one can prove that $0\leq (y_1-y_3)\leq y_1$. Hence,
	\[
	\abs{z_1-tz_2}^2=(x_1-x_3)^2+\abs{\gamma}^2(y_1-y_3)^2\leq x_1^2+\abs{\gamma}^2y_1^2= \abs{z_1}.
	\]
	We have that the inequality is an equality only if $x_1-x_3=x_1$ and $y_1-y_3=y_1$. This is absurd since $z_2\neq0$. So, $\abs{z_1-tz_2}<\abs{z_1}$.
\end{proof}
Let $\p$ be a good prime and consider the ideal $\Ann_\p(P)$ as defined in Definition \ref{Ip}. If $\p$ is a primitive divisor for $B_\alpha$, then $B_\beta$ is divisible by $\p$ only if $\beta$ can be written as $\beta=\alpha+i$ for $i\in \Ann_\p(P)$. Hence,
\[
\Norm{\alpha}=\min\{\Norm{\alpha+i}: i\in \Ann_\p(P)\}.
\]
We want to understand the relation between $\Norm{\alpha}$ and $\Norm{\Ann_\p(P)}$. We will show that $\Norm{\alpha}$ cannot be much larger (in some sense) than $\Norm{\Ann_\p(P)}$.
\begin{lemma}\label{C}
	Let $I$ be an integral $\OO$-ideal. Let $\alpha\in \OO$ be so that
	\[
	\Norm{\alpha}=\min\{\Norm{\alpha+i}: i\in I\}.
	\] 
	Then, given $i_1$ and $i_2$ in $I$ with $\Norm{i_1}>\Norm{i_2}$, we have
	\[
	\Norm{i_1}-\Norm{i_2}>C_8\Norm{\alpha},
	\]
	with $C_8$ that depends only on $\OO$.
\end{lemma}
\begin{proof}
	If $i\in I$, then $\Norm{I}$ divides $\Norm{i}$. This follows from the fact that the degree is multiplicative and $I\mid (i)$. So, $\Norm{i_1}$ and $\Norm{i_2}$ are both multiples of $\Norm{I}$ and then
	\begin{equation}\label{eqnorm}
	\Norm{i_1}-\Norm{i_2}\geq\Norm{I}.
	\end{equation}
	If $\OO=\Z$, the Lemma is easy. Assuming $\alpha>0$ and $I=(i)$ with $i>0$, then $\alpha<i$, otherwise $\Norm{\alpha-i}<\Norm{\alpha}$. So, $\Norm{I}>\Norm{\alpha}$. The case $\alpha<0$ is analogous. Since $\Z$ is a PID, we conclude that the lemma holds with $C_8=1$. 
	
	So, we assume that $L=\OO\otimes_\Z \Q$ is an imaginary quadratic field.
	Let $i\neq 0$ be such that
	\[
	\Norm{i}=\min\{\Norm{j}: j\in I\setminus\{0\}\}
	\]
	and then it is well-known that there exists a constant $C_9$ depending only on $L$ so that $\Norm{i}\leq C_9\Norm{I}$ (see \cite[Lemma 6.2]{neuk}). Let $\gamma\in \OO$ be such that $\gamma^2\in\Z<0$. The existence of $\gamma$ follows from the fact that $L$ is an imaginary quadratic field. Put $z_1=\alpha$ and $z_2=i$. Fixing an embedding $L\hookrightarrow \C$, we can assume that $z_1$ and $z_2$ are in $\C$. For a contradiction, assume $2\abs{\gamma}\abs{z_2}<\abs{z_1}$. Therefore, from the previous lemma, 
	\[
	\abs{z_1-tz_2}<\abs{z_1}
	\]
	where $t \in\{1,-1,\gamma,-\gamma\}\subseteq \OO$. Hence, $t i\in I$ and
	\[
	\Norm{\alpha-t i}<\Norm{\alpha}.
	\]
		Indeed, using the definition of the degree, we have $\Norm{\cdot}=\abs{\cdot}^2$.
	This is absurd for the hypothesis $\Norm{\alpha}=\min\{\Norm{\alpha+i}: i\in I\}.$ So, $2\abs{\gamma}\abs{z_2}\geq\abs{z_1}$ and then
	\[
	4\Norm{\gamma}\Norm{i}\geq \Norm{\alpha}.
	\]
	Therefore,
	\[
	\Norm{\alpha}\leq 4\Norm{\gamma}\Norm{i}\leq 4C_9\Norm{\gamma}\Norm{I}.
	\]
	Finally, using (\ref{eqnorm}),
	\[
	\Norm{i_1}-\Norm{i_2}\geq\Norm{I}\geq \frac{\Norm{\alpha}}{4C_9\Norm{\gamma}}.
	\]
\end{proof}
Before proving Theorem \ref{Thm2}, we need another technical lemma.
\begin{lemma}\label{K}
	Take $f$, $g$, $\alpha$ and $\beta$ in $\OO$. Assume $\Norm{\beta}<\Norm{\alpha}$ and $\Norm{\beta g+f}\geq \Norm{\alpha g+f}$. Then,
	\[
	0\leq \Norm{\beta g+f}-\Norm{\alpha g+f}\leq C_{10}\sqrt{\Norm{\alpha}}
	\]
	for $C_{10}$ that depends only on $f$ and $g$.
\end{lemma}
\begin{proof}
	Given $\gamma\in \OO$, we put $\overline{\gamma}$ as the conjugate of $\gamma$. Thus, $\Norm{\gamma}=\gamma\overline{\gamma}$. Denote with $\Re{x}$ the real part of the complex number $x$. Here, we use again the embedding $L \hookrightarrow \C$. So,
	\begin{align*}
		0&\leq \Norm{\beta g+f}-\Norm{\alpha g+f}\\&=(\beta g+f)\overline{(\beta g+f)}-(\alpha g+f)\overline{(\alpha g+f)}\\&=\Norm{g}(\Norm{\beta}-\Norm{\alpha})+(\beta-\alpha)g\overline{f}+\overline{(\beta-\alpha)g\overline{f}}\\&< (\beta-\alpha)g\overline{f}+\overline{(\beta-\alpha)g\overline{f}} \\&\leq \abs{2\Re{(\beta-\alpha)g\overline{f}}}\\&\leq \abs{2(\beta-\alpha)fg}\\&\leq 4\abs{\alpha}\abs{fg}.
	\end{align*}
	We conclude using that $\Norm{\alpha}=\abs{\alpha}^2$.
\end{proof}
Now, we are ready to prove the theorem. We want to show that, if $f(P)=g(Q)$, then $B_\alpha(P,Q)$ has a primitive divisor for all but finitely many terms.
\begin{proof}[Proof of Theorem \ref{Thm2}]
	Suppose that $f=0$. Then, $g(Q)=O$ and so $Q$ is a torsion point. Then, this is exactly Theorem \ref{Thm1}.
	
	Hence, we assume that $f\neq 0$. Take $P'\in E(\overline{\Q})$ so that $g(P')=P$. Here we are using the hypothesis $g\neq 0$. Then, \[g(f(P')-Q)=f(g(P'))-g(Q)=f(P)-g(Q)=O\] and so $f(P')=Q-T$ with $T$ a torsion point. So, \[\alpha P+Q=\alpha gP'+fP'+T=(\alpha g+f)P'+T\] and then $B_\alpha(P,Q)=B_{\alpha g+f}(P',T)$. Enlarging $K$, we can assume that $P'$ and $T$ belong to $E(K)$. Observe that $\alpha g+f\in \OO$ since $f$ and $g$ are in $\OO$. Thanks to Theorem \ref{Thm1}, for all but finitely many $\alpha$, there exists a prime $\p$ that is a primitive divisor for $B_{\alpha g+f}(P',T)$ in the sequence $\{B_\gamma(P',T)\}_{\gamma\in \OO}$. 
	
	Suppose that $\p$ is not a primitive divisor for $B_\alpha(P,Q)$. We will show that this happens only for $\Norm{\alpha}$ small. Since $\p$ is not a primitive divisor, then there exists $\beta$ so that $\p$ divides $B_\beta(P,Q)$, $\Norm{\beta}<\Norm{\alpha}$ and $B_\beta(P,Q)\neq 0$. Observe that \[B_{\beta g+f}(P',T)=B_\beta(P,Q)\neq 0.\] Since $\p$ is a primitive divisor for $B_{\alpha g+f}(P',T)$, we have $\Norm{\alpha g+f}\leq \Norm{\beta g+f}$. Thanks to Lemma \ref{K}
	\begin{equation}\label{eq1}
		0\leq \Norm{\beta g+f}-\Norm{\alpha g+f}\leq C_{10}\sqrt{\Norm{\alpha}}.
	\end{equation}
	Recall that $C_{10}$ is defined in Lemma \ref{K}.
	Take $\Ann_{\p}(P')$ as defined in Definition \ref{Ip}. Thus, $\delta P'+T\equiv O \mod{\p}$ if and only if $\delta \in \alpha g+f+\Ann_{\p}(P')$ since
	\[
	(\delta-\alpha g-f)P'\equiv \delta P'+T-((\alpha g+f)P'+T)\equiv O\mod{\p}.
	\]
	
	If $T=O$, then $\Ann_{\p}(P')=(\alpha g+f)$. Thus, for $i_1,i_2\in \Ann_{\p}(P')$ with $\Norm{i_1}>\Norm{i_2}$,
		\begin{equation}\label{TO}
		\Norm{i_1}-\Norm{i_2}\geq \Norm{\alpha g+f}.
		\end{equation}
	
	Assume $T\neq O$. Observe that $P'$ is a non-torsion point since $g(P')=P$. Then $\gamma P'+T\neq O$ for every $\gamma\in \OO$ and therefore $B_\gamma(P',T)\neq 0$. Since $\p$ is a primitive divisor for $B_{\alpha g+f}(P',T)$,
	\[
	\Norm{\alpha g+f}=\min\{\Norm{\alpha g+f+i} : i\in \Ann_{\p}(P')\}.
	\] Thus, we are in the hypothesis of Lemma \ref{C}. So, for $i_1,i_2\in \Ann_{\p}(P')$ with $\Norm{i_1}>\Norm{i_2}$,
	\begin{equation}\label{eqC8}
	\Norm{i_1}-\Norm{i_2}\geq C_8\Norm{\alpha g+f}
	\end{equation}
	where $C_8$ is defined in Lemma \ref{C}. Combining with Equation (\ref{TO}) and potentially decreasing $C_8$, we have that
		\[
		\Norm{i_1}-\Norm{i_2}\geq C_8\Norm{\alpha g+f}
		\]
		even in the case $T=O$.
		 
	Fix $s\in \OO$ such that $sT=O$.
	Observe that $s(\alpha g+f)$ and $s(\beta g+f)$ both belong to $\Ann_{\p}(P')$ since 
	\[
	s(\alpha g+f) P'\equiv s(\alpha g+f) P'+sT\equiv s((\alpha g+f) P'+T)\equiv O\mod{\p}.
	\] Using Lemma \ref{C},
	\[
	\Norm{s(\beta g+f)}-\Norm{s(\alpha g+f)}\geq C_8\Norm{\alpha g+f}.
	\]
	So, using (\ref{eq1}), we have
	\begin{equation}\label{K'alpha}
	C_{10}\sqrt{\Norm{\alpha}}\geq \Norm{\beta g+f}-\Norm{\alpha g+f}\geq \frac{C_8}{\Norm{s}}\Norm{\alpha g+f}=O(\Norm{\alpha}).
	\end{equation}
	For $\Norm{\alpha}$ large enough, Equation (\ref{K'alpha}) cannot hold. Thus, $\p$ is a primitive divisor for $B_{\alpha}(P,Q)$ for $\Norm{\alpha}$ large enough. In conclusion, for all but finitely many $\alpha$, the ideal $B_\alpha(P,Q)$ has a primitive divisor.
\end{proof}
\begin{proof}[Proof of Theorem \ref{Thm3}]
	Since $\Z$ is a Dedekind domain, the theorem follows from Theorem \ref{Thm1}. 
\end{proof}
Now, we want to use Theorem \ref{Thm2} to prove a corollary for the sequence $\{B_n(P,Q)\}_{n\in \Z}$. Thanks to Theorem \ref{Thm3}, we know that if $Q$ is a torsion point, then $B_n(P,Q)$ has a primitive divisor for all but finitely many terms. We will prove a similar result, without the hypothesis that $Q$ is a torsion point. In order to do so, we need some additional hypotheses on the curve $E$. 
\begin{corollary}\label{corshifter}
	Let $K$ be a number field. Take $E$ an elliptic curve such that one of the following holds:
	\begin{itemize}
		\item $E(K)$ has rank $1$ or,
		\item $E$ has CM, $E(K)$ has rank $2$, $\End(E)$ is a Dedekind domain, and $L=\End(E)\otimes_\Z\Q\subseteq K$.
	\end{itemize} Take $P$ a non-torsion point and $Q$ in $E(K)$. Then, $B_n(P,Q)$ has a primitive divisor for all but finitely many $n$.
\end{corollary}
\begin{proof}
	Suppose that $E$ has rank $1$. Let $R$ be the generator of the non-torsion points and so $P=aR+T_1$ and $Q=bR+T_2$ with $T_1$ and $T_2$ two torsion points. Observe that $a\neq 0$ since $P$ is a non-torsion point. Take $n>0$ such that $nT_1=nT_2=O$ and then
	\[
	bnP=bnaR+bnT_1=anQ-anT_2+bnT_1=anQ.
	\]
	So, we take $f=bn$ and $g=an$. Observe that $g\neq 0$ since $a\neq 0$ and $n\neq 0$. We conclude applying Theorem \ref{Thm2} to $\OO=\Z$, since we have that $f(P)=g(Q)$ with $f,g\in \OO$. 
	
	Suppose now that $E$ has rank $2$ and has complex multiplication. Let $L=\End(E)\otimes_\Z\Q$, which is a quadratic field. Therefore, $\End(E)$ is an order in $\OO_L$. If we prove that $f(P)=g(Q)$ for $f,g\in \End(E)$, then we can apply Theorem \ref{Thm2} to conclude. 
	
	Let $R_1$ be a non-torsion point and $\gamma\in \End(E)\setminus \Z$. Then, $\gamma(R_1)-mR_1$ is not a torsion point for every $m\in \Z$. Indeed, if it is a torsion point of order $n$, then $n(\gamma-m)\neq 0$ has kernel every multiple of $R_1$, that is absurd since $R_1$ has infinite order. 
	
	Let $E^{\text{tors}}(K)$ be the subgroup of the torsion points and $E'=E(K)/E^{\text{tors}}(K)$, which is a free group of rank $2$. So, $\gamma R_1-mR_1\neq 0$ in $E'$. Suppose that $R_1$ and $R_2$ are the generators of $E'$. Take $\gamma \in \End(E)\setminus \Z$ and then \[\gamma R_1=aR_1+bR_2,\] with $b\neq 0$ (the equality is in $E'$). Here we are using the hypothesis $L=\End(E)\otimes_\Z\Q\subseteq K$ since we need that $\gamma R_1\in E(K)$. Take $P_1$ and $Q_1$ in $E'$ as the image of $P$ and $Q$ under the quotient. So, $P_1=xR_1+y R_2$ with $x,y\in \Z$ (the equality is in $E'$). Hence,
	\[
	bP_1=bxR_1+y(bR_2)=(bx+y\gamma -ya) R_1
	\]
	and so $bP_1=\alpha R_1$ for $\alpha\in \End(E)$. Since $b\neq 0$, we have $bP_1\neq O$ using that $P$ is a non-torsion point. So, $\alpha\neq 0$. In the same way, $bQ_1=\beta R_1$ for $\beta\in \End(E)$.
	Observe that, in $E'$,
	\[
	\beta bP_1=\beta \alpha R_1=\alpha b Q_1.
	\]
	Hence, returning to $E$,
	\[
	\beta bP-\alpha bQ=T,
	\]for $T\in E^{\text{tors}}(K)$. Take $n>0$ such that $nT=O$ and then
	\[
	n\beta b P-n\alpha bQ=nT=O.
	\]
	So, putting $f=n\beta b$ and $g=nb\alpha$, we have $f(P)=g(Q)$. Observe that $g\neq 0$ since $\alpha\neq 0$, $b\neq 0$, and $n\neq 0$. We conclude by applying Theorem \ref{Thm2}.
\end{proof}
\begin{example}
Let $E$ be the elliptic curve defined by the equation $y^2=x^3-2x$ in $K=\Q(i)$. Then, checking the database \cite{lmfdb}, we have $\End(E)=\Z[i]$ and $E(K)$ has rank $2$. Moreover $\End(E)\otimes_\Z\Q= K$. The ring $\Z[i]$ is a Dedekind domain and so we are in the hypothesis of the previous corollary. Hence, for every non-torsion $P$ and every $Q$ in $E(K)$, $B_n(P,Q)$ has a primitive divisor for all but finitely many term. This example shows that the previous corollary is actually more general than Theorem \ref{Thm3}.
\end{example}
\section*{Acknowledgements}
The author wants to gratefully thank Professor Marco Streng for proposing the problem, for carefully reading a preliminary version of the paper, and for the many useful comments. The author also wants to thank the anonymous referee for valuable comments which greatly have improved this article. 
			\normalsize
	\bibliographystyle{plain}
	\bibliography{biblio}
	MATTEO VERZOBIO, UNIVERSIT\'A DI PISA, DIPARTIMENTO DI\\ MATEMATICA, LARGO BRUNO PONTECORVO 5, PISA, ITALY\\
	\textit{E-mail address}: matteo.verzobio@gmail.com
\end{document}